\documentclass[onefignum,onetabnum]{siamart171218}

\usepackage{amsmath}
\usepackage{enumerate}
\usepackage{mathrsfs}
\usepackage{color}
\usepackage{tikz}
\usepackage{mathtools}
\usepackage{breqn}
\usepackage{colortbl}
\usepackage{subfigure}

\newcommand{\defword}[1]{\textcolor{red}{\em #1}}

\DeclareMathOperator*{\argmin}{\arg\!\min}

\newtheorem{example}[theorem]{Example}

\newtheorem{question}[theorem]{Question}
\newtheorem{assumption}[theorem]{Assumption}

\usepackage{lipsum}
\usepackage{amsfonts}
\usepackage{graphicx}
\usepackage{epstopdf}
\usepackage{algorithmic}
\ifpdf
  \DeclareGraphicsExtensions{.eps,.pdf,.png,.jpg}
\else
  \DeclareGraphicsExtensions{.eps}
\fi

\newsiamremark{remark}{Remark}
\newsiamremark{hypothesis}{Hypothesis}
\crefname{hypothesis}{Hypothesis}{Hypotheses}
\newsiamthm{claim}{Claim}

\headers{Multistationarity of One-Dim Reaction Networks}{Kexin Lin, Xiaoxian Tang,  and Zhishuo Zhang}

\title{Multistationarity of Reaction Networks\\ with One-Dimensional Stoichiometric Subspaces\thanks{Submitted to the editors DATE.
\funding{XT was funded by the NSFC12001029.}
}}

\author{
Kexin Lin\thanks{School of Mathematical Sciences, Beihang University, Beijing,  China
  (\email{19377317@buaa.edu.cn}).}
\and
Xiaoxian Tang\thanks{School of Mathematical Sciences, Beihang University, Beijing,  China
  (\email{xiaoxian@buaa.edu.cn}, \url{https://sites.google.com/site/rootclassification/}).}
\and Zhishuo Zhang\thanks{School of Mathematical Sciences, Beihang University, Beijing,  China
  (\email{793008192@buaa.edu.cn}).}
}

\usepackage{amsopn}

\ifpdf
\hypersetup{
  pdftitle={SMALL NETWORKS WITH ONE-DIMENSIONAL STOICHIOMETRIC SUBSPACES},
  pdfauthor={}
}
\fi

\begin{document}

\maketitle
\begin{abstract}
We study the multistationarity for the reaction networks with one-dimensional stoichiometric subspaces, and we focus on the networks admitting finitely many positive steady states.
We prove that if  a network admits multistationarity, then network has an embedded one-species network with arrow diagram $(\rightarrow,\leftarrow)$ and another with arrow diagram $(\leftarrow,\rightarrow)$. The inverse is also true if there exist
 two reactions in the network such that the subnetwork consisting of the two reactions  admits at least one and  finitely many positive steady states.
 We also prove that if a network admits at least three positive steady states, then
it contains at least three bi-arrow diagrams. More than that, we completely characterize the bi-reaction networks that admit at least three positive
 steady states.
 \end{abstract}



\section{Introduction}\label{sec:intro}

This work is motivated by the multistability problem of dynamical systems arising from biochemical reaction networks (under mass-action kinetics): for which rate constants, a network has at least two stable positive steady states in the same stoichiometric compatibility class?  Multistability is a frontier topic in mathematical biology since it widely exists in
the decision-making process and switch-like behavior in cellular signaling (e.g., \cite{BF2001, XF2003, CTF2006, SN2020,SH2021}). Multistability problem is known to be a special real quantifier elimination problem
so it is challenging to solve it by the computational tools in real algebraic geometry (e.g., \cite{HTX2015, B2017}). Given a network, one way to find multistability is to look for a witness for (nondegenerate) multistationarity, i.e.,
a choice of parameters (rate constants and total constants) such that the network has at least two positive (nondegenerate) steady states. In practice, if the number of positive nondegenerate steady states  is large enough, one can usually obtain at least two stable ones (e.g., \cite{OSTT2019,TF2020,CH2021}). Deciding multistationarity or computing the witnesses for multistationarity is not easy neither but there exists a collection of nice methods (e.g., \cite{ShinarFeinberg2012,CF2012,  signs, joshi2013complete}).  For instance, one typical approach is to check if the determinant of a certain Jacobian matrix changes sign \cite{CF2005, BP,WiufFeliu_powerlaw,CFMW}. Especially, for the networks with binomial steady-state equations, deciding multistationary  can be unexpectedly simple \cite{DMST2019,SadeghimaneshFeliu2019}.


Since multistability or nondegenerate multistationarity can be lifted from small subnetworks to the corresponding large networks \cite{JS13, BP16},  criteria of (nondegenerate)
multistationarity for small networks with only one species or up to two reactions (possibly reversible) are studied in \cite{Joshi:Shiu:Multistationary, shiu-dewolff}. At the end of \cite{Joshi:Shiu:Multistationary}, the authors wonder if their results can be extended to more general networks with one-dimensional stoichiometric subpaces (note here, for a network with two reactions, if it admits multistationarity, then it has a one-dimensional stoichiometric subspace \cite{Joshi:Shiu:Multistationary}). More specifically, they proposed the following question:

\begin{question}\label{ques:6.1}\cite[Question 6.1]{Joshi:Shiu:Multistationary}
  Consider a network $G$ with a one-dimensional stoichiometric subspace. For $G$ to be multistationary, is it necessary for $G$ to have an embedded one-species network with arrow diagram $(\rightarrow,\leftarrow)$ and another with arrow diagram $(\leftarrow,\rightarrow)$? Is it sufficient?
\end{question}
It is remarkable that for the networks with one-dimensional stoichiometric subspaces, multistationarity is equivalent to nondegenerate multistationarity if the maximum number of positive steady states is finite (see Theorem \ref{thm:nc}). It is also worth mentioning that if a network with a one-dimensional stoichiometric subspace admits multistability, then it admits at least three positive steady states (e.g., \cite[Theorem 3.4]{txzs2020}).  So, it is also important to extend the results in \cite{Joshi:Shiu:Multistationary} to the networks admitting at least three positive steady states.

In this paper, we study the multistationarity problem for the networks with one-dimensional stoichiometric subspaces, and we focus on the networks admitting finitely many positive steady states. We answer Question \ref{ques:6.1} and extend the problem by the following results.
\begin{itemize}
\item[(1)]If a network admits multistationarity, then the network has an embedded one-species network with arrow diagram $(\rightarrow,\leftarrow)$ and another with arrow diagram $(\leftarrow,\rightarrow)$ (Theorem \ref{thm:Q6.1}).
    The inverse is also true if we additionally assume that a subnetwork consisting of two reactions from the original network admits at least one and finitely many positive steady states (Theorem \ref{thm:inverse}).
\item[(2)]If a network admits at least three positive steady states, then
it contains at least three bi-arrow diagrams (Theorem \ref{thm:ad3} and Corollary \ref{cry:ad3}).
\item[(3)]We completely characterize the stoichiometric coefficients of the bi-reaction networks that admit at least three positive steady states (Theorem \ref{thm:main}).
\end{itemize}
We provide Example \ref{ex:counter} for illustrating that the inverse of Theorem \ref{thm:Q6.1} might not be true in general.
We also remark that for the networks admitting infinitely many positive steady states, it is not necessary for a multistationary network to contain the pair of arrow diagrams $(\rightarrow,\leftarrow)$ and $(\leftarrow,\rightarrow)$ (see Example \ref{ex}). 
Note that Theorem \ref{thm:Q6.1} implies that if a network admits at least two positive steady states, then it contains at least two bi-arrow diagrams (see Corollary \ref{cry:ad2}). So, Theorem \ref{thm:ad3} can be viewed as a generalization of Theorem \ref{thm:Q6.1} (to see this, one can also compare Theorem \ref{thm:Q6.1} and Corollary \ref{cry:ad3}).
We remark that admitting three bi-arrow diagrams is also a necessary condition for multistability  (Corollary \ref{cry:ad3stable}).
Theorem \ref{thm:main} shows that
the inverse of Theorem \ref{thm:ad3} is not true in general.
The proof of Theorem \ref{thm:main} is constructive, which provide one way for getting witnesses for three positive steady states.
Theorem \ref{thm:main} also shows that it is possible for us to enumerate the bi-reaction networks admitting at least three positive steady states.
From these networks, as a future direction,  one might characterize  the bi-reaction networks admitting multistability.

The rest of this paper is organized as follows.
In Section~\ref{sec:back}, we recall the basic definitions and notions for the
mass-action kinetics systems  arising  from reaction networks.
In Section~\ref{sec:dim1}, we prepare the steady-state system augmented with the conservation laws for the networks with one-dimensional stoichiometric subspaces. We also make some assumptions for the rest of the discussion.
In Section~\ref{sec:multi}, we prove a necessary  (Theorem \ref{thm:Q6.1}) and a sufficient condition (Theorem \ref{thm:inverse}) for a network with a one-dimensional stoichiometric subspace to admit multistationarity.
In Section~\ref{sec:nec}, we prove a necessary condition (Theorem \ref{thm:ad3}) for a network with a one-dimensional stoichiometric subspace to admit  at least three positive steady states.
In Section~\ref{sec:bi}, we completely characterize the bi-reaction networks that admit at least three positive steady states (Theorem \ref{thm:main}).
Finally, we end this paper with some future directions inspired by  the main results, see Section \ref{sec:dis}.

\section{Background}\label{sec:back}

\subsection{Chemical reaction networks}\label{sec:pre}



In this paper, we follow the standard notions on reaction networks used in \cite{txzs2020}, see \cite{CFMW, Joshi:Shiu:Multistationary} for more details.
A \defword{reaction network} $G$  (or \defword{network} for short) consists of a set of $m$ reactions:
\begin{align}\label{eq:network}
\alpha_{1j}X_1 +
 \dots +
\alpha_{sj}X_s
~ \xrightarrow{\kappa_j} ~
\beta_{1j}X_1 +
 \dots +
\beta_{sj}X_s,
 \;
    {\rm for}~
	j=1,2, \ldots, m,
\end{align}
where $X_1, \ldots, X_s$ denote $s$ \defword{species},  the \defword{stoichiometric coefficients} $\alpha_{ij}$ and $\beta_{ij}$ are non-negative integers, each $\kappa_j \in \mathbb R_{>0}$ is a \defword{rate constant} corresponding to the
$j$-th reaction,
and
we assume that
\begin{align}\label{eq:netcon}
\text{for every}\; j\in \{1, \ldots, m\},\; (\alpha_{1j},\ldots,\alpha_{sj})\neq (\beta_{1j},\ldots,\beta_{sj}).
\end{align}
The
\defword{stoichiometric matrix} of
$G$,
denoted by ${\mathcal N}$, is the
 $s\times m$ matrix with
$(i, j)$-entry equal to $\beta_{ij}-\alpha_{ij}$.
The \defword{stoichiometric subspace}, denoted by $S$, is the real vector space spanned by the column vectors of ${\mathcal N}$.

The concentrations of the species $X_1,X_2, \ldots, X_s$ are denoted by $x_1, x_2, \ldots, x_s$, respectively. Note that
$x_i$ can be considered as a function in the time variable.
Under the assumption of mass-action kinetics, we describe how these concentrations change  in time by following system of ordinary differential equations
(ODEs):
\begin{align}\label{eq:sys}
\dot{x}~=~(f_1(\kappa; x), \ldots, f_s(\kappa; x))^{\top}~:=~{\mathcal N}\cdot \begin{pmatrix}
\kappa_1 \, x_1^{\alpha_{11}}
		x_2^{\alpha_{21}}
		\cdots x_s^{\alpha_{s1}} \\
\kappa_2 \, x_1^{\alpha_{12}}
		x_2^{\alpha_{22}}
		\cdots x_s^{\alpha_{s2}} \\
		\vdots \\
\kappa_m \, x_1^{\alpha_{1m}}
		x_2^{\alpha_{2m}}
		\cdots x_s^{\alpha_{sm}} \\
\end{pmatrix}~,
\end{align}
where $x$ denotes the vector $(x_1, x_2, \ldots, x_s)$,
and $\kappa$ denotes the vector $(\kappa_1, \ldots, \kappa_m)$.
 Note that for every $i\in \{1, \ldots, s\}$, $f_{i}(\kappa;x)$ is a polynomial  in $\mathbb Q[\kappa, x]$.

A \defword{conservation-law matrix} of $G$, denoted by $W$, is any row-reduced $d\times s$ matrix (here, $d:=s-{\rm rank}({\mathcal N})$), whose rows form a basis of $S^{\perp}$. Note that ${\rm rank}(W)=d$. Especially, if the stoichiometric subspace of $G$ is one-dimensional, then
${\rm rank}({\mathcal N})=1$ and ${\rm rank}(W)=s-1$.
Note that the system~\eqref{eq:sys} satisfies $W \dot x =0$,  and 
any trajectory $x(t)$ beginning at a nonnegative vector $x(0)=x^0 \in
\mathbb{R}^s_{> 0}$ remains, for all positive time,
 in the following \defword{stoichiometric compatibility class} with respect to the  \defword{total-constant vector} $c:= W x^0 \in {\mathbb R}^d$:
\begin{align}\label{eq:pc}
{\mathcal P}_c~:=~ \{x\in {\mathbb R}_{\geq 0}^s \mid Wx=c\}.~
\end{align}

\subsection{Multistationarity and multistability}\label{sec:mm}
For a given rate-constant vector $\kappa\in \mathbb{R}_{>0}^m$, a \defword{steady state} 
of~\eqref{eq:sys} is a concentration vector
$x^* \in \mathbb{R}_{\geq 0}^s$ such that $f_1(\kappa, x^*)=\cdots=f_s(\kappa, x^*)=0$, where  $f_1, \ldots, f_s$ are on the
right-hand side of the
ODEs~\eqref{eq:sys}.
If all coordinates of a steady state $x^*$ are strictly positive (i.e., $x^*\in \mathbb{R}_{> 0}^s$), then we call $x^*$ a \defword{positive steady state}.
We say a steady state $x^*$ is \defword{nondegenerate} if
${\rm im}\left({\rm Jac}_f (x^*)|_{S}\right)=S$,
where ${\rm Jac}_f(x^*)$ denotes the Jacobian matrix of $f$, with respect to $x$, at $x^*$.
A steady state $x^*$ is \defword{exponentially stable} (or simply \defword{stable})
if it is nondegenerate, and  all non-zero eigenvalues of ${\rm Jac}_f(x^*)$ have negative real parts.
Note that if a steady state is exponentially stable, then it is locally asymptotically stable \cite{P2001}.




Suppose $N\in {\mathbb Z}_{\geq 0}$. We say a  network  \defword{admits $N$ (nondegenerate) positive steady states}  if there exist a rate-constant
vector $\kappa\in  \mathbb{R}_{>0}^m$ and a total-constant vector $c\in  \mathbb{R}^d$ such that it has $N$ (nondegenerate) positive steady states  in the stoichiometric compatibility class ${\mathcal P}_c$.
Similarly, we say a  network  \defword{admits $N$ stable positive steady states}  if there exist a rate-constant
vector $\kappa\in  \mathbb{R}_{>0}^m$ and a total-constant vector $c\in  \mathbb{R}^d$ such that  it has $N$ stable positive steady states  in ${\mathcal P}_c$.

The \defword{maximum number of positive steady states} of a network $G$ is
{\footnotesize
\[cap_{pos}(G)\;:=\;\max\{N\in {\mathbb Z}_{\geq 0}\cup \{+\infty\}|G \;\text{admits}\; N\; \text{positive steady states}\}.\]
}
Similarly, we define
{\footnotesize
\[cap_{nondeg}(G)\;:=\;\max\{N\in {\mathbb Z}_{\geq 0}\cup \{+\infty\}|G \;\text{admits}\; N\; \text{nondegenerate positive steady states}\}\]
}
and
{\footnotesize
\[cap_{stab}(G)\;:=\;\max\{N\in {\mathbb Z}_{\geq 0}\cup \{+\infty\}|G \;\text{admits}\; N\; \text{stable positive steady states}\}.\]
}
We say a network admits \defword{multistationarity} if $cap_{pos}(G)\geq 2$.
We say a network admits \defword{nondegenerate multistationarity} if $cap_{nondeg}(G)\geq 2$.
We say a network admits \defword{multistability} if $cap_{stab}(G)\geq 2$.
\begin{theorem}\label{thm:nc}\cite[Theorem 6.1]{txzs2020}
Given a network $G$ with a one-dimensional stoichiometric subspace,
if $cap_{pos}(G)<+\infty$, then $cap_{nondeg}(G)=cap_{pos}(G)$.
\end{theorem}



\section{Steady states of networks with one-dimensional stoichiometric subspaces}\label{sec:dim1}

\begin{assumption}\label{assumption1}
For any network $G$ with reactions defined in \eqref{eq:network},
by  \eqref{eq:netcon},
we know $(\alpha_{11},\ldots, \alpha_{s1})\neq (\beta_{11},\ldots, \beta_{s1})$.
 Without loss of generality, we would assume 
  $\beta_{11}-\alpha_{11}\neq 0$ throughout this paper.
\end{assumption}

 If the stoichiometric subspace of a network $G$ \eqref{eq:network} is one-dimensional, then under Assumption \ref{assumption1}, for every
$j\in \{1,\ldots,m\}$,
 there exists $\lambda_j \in {\mathbb R}$ such that
\begin{align}\label{eq:scalar}
\left(
\begin{array}{c}
\beta_{1j}-\alpha_{1j}\\
\vdots\\
\beta_{sj}-\alpha_{sj}
\end{array}
\right)\;=\;
\lambda_j\left(
\begin{array}{c}
\beta_{11}-\alpha_{11}\\
\vdots\\
\beta_{s1}-\alpha_{s1}
\end{array}
\right).
\end{align}
Note here, we have $\lambda_1=1$, and by \eqref{eq:netcon}, we have
$\lambda_j\neq 0$ for all $j\in \{2, \ldots, m\}$.

By substituting
\eqref{eq:scalar} into $f_1, \ldots, f_s$ in \eqref{eq:sys},
we have
\begin{align}\label{eq:f}
f_i \;=\;\left(\beta_{i1}-\alpha_{i1}\right) \sum\limits_{j=1}^m\lambda_j\kappa_j\prod\limits_{k=1}^s x_k^{\alpha_{kj}}, \;\; i=1, \ldots, s.
\end{align}
We define the steady-state system augmented with the conservation laws:
\begin{align}
h_1 ~&:=~ f_1\;=\;\left(\beta_{11}-\alpha_{11}\right)\sum\limits_{j=1}^m
\lambda_j\kappa_j\prod\limits_{k=1}^s x_k^{\alpha_{kj}} \label{eq:h1}\\
h_i ~&:= ~(\beta_{i1}-\alpha_{i1})x_1 - (\beta_{11}-\alpha_{11})x_i - c_{i-1}, \;\; i=2, \ldots, s.\label{eq:h}
\end{align}

Clearly, if for a rate-constant $\kappa^*$,
$x^*$ is a steady state in ${\mathcal P}_{c^*}$, then
$x^*$ is a common solution to the equations $h_1(\kappa^*, c^*, x^*)=\ldots=h_s(\kappa^*, c^*, x^*)=0$.

\section{Multistationarity}\label{sec:multi}

The main results of this section  are
Theorem \ref{thm:Q6.1} and Theorem \ref{thm:inverse}, which are proved in Sections
\ref{sec:multip}--\ref{sec:multip2}.  We first recall the definitions for embedded subnetworks and arrow diagrams, which are needed in the statement of the theorems. Then, we will answer Question \ref{ques:6.1}  by Theorem \ref{thm:Q6.1}, Example \ref{ex}, Example \ref{ex:counter}, and Theorem \ref{thm:inverse}.

\begin{definition}\cite[Definition 2.11]{Joshi:Shiu:Multistationary}
 Given a network $G$, the embedded subnetwork of $G$ is a network constructed by removing certain reactions and reactants of $G$.
\end{definition}

\begin{definition}\label{def:arrow}\cite[Definition 3.3]{Joshi:Shiu:Multistationary}
Let $G$ be a reaction network that contains only one species $X_1$. Thus, each reaction of $G$ has the form $ aX_1 \rightarrow \beta X_1$, where $ a, \beta{\ge}0 $ and $ a {\neq} \beta$. Let $\ell$ be the number of (distinct) reactant complexes, and let $a_{1}<a_{2}<...<a_{\ell}$ be their stoichiometric coefficients. The \defword{arrow diagram} of $G$, denoted by $\rho = (\rho_1,...,\rho_{\ell})$, is the element of $\{ \rightarrow,\leftarrow,\leftarrow\!\!\!\!\bullet\!\!\!\!\rightarrow \}^{\ell}$ defined by:
	\begin{equation}
		\rho_{i}=  \left\{
		\begin{array}{ll}
			\rightarrow \;\; & \text{if for all reactions $a_{i}X_1{\rightarrow}\beta X_1$ in $G$, it is the case that $\beta{>}a_i$} \\
			\leftarrow \;\; & \text{if for all reactions $a_{i}X_1{\rightarrow}\beta X_1$ in $G$, it is the case that $\beta{<}a_i$} \\
			\leftarrow\!\!\!\!\bullet\!\!\!\!\rightarrow & \text{otherwise.}
		\end{array}
		\right.
	\end{equation}
\end{definition}



\begin{theorem}\label{thm:Q6.1}
Given a network $G$ \eqref{eq:network} with a one-dimensional stoichiometric subspace, suppose $cap_{pos}(G)<+\infty$.
If the network $G$ admits multistationarity, then  $G$ has an embedded one-species network with arrow diagram $(\leftarrow, \rightarrow)$ and another with arrow diagram $(\rightarrow, \leftarrow)$.
\end{theorem}
\begin{example}
  Consider the following network:
  $$G:\begin{aligned}
      X_1+2X_2&\rightarrow X_2\\
      2X_1&\rightarrow 3X_1+X_2\\
      2X_2&\rightarrow X_1+3X_2
    \end{aligned}$$
Let $\kappa_1=1,~\kappa_2=9,~\kappa_3=1$, and $c_1=0.5$. By solving the equations $h_1(x)=h_2(x)=0$ (see \eqref{eq:h1}--\eqref{eq:h}), we see that the network has two positive steady states: $$x^{(1)}=(0.30806, 0.80806), ~~x^{(2)}=(6.8111,7.3111). $$ So, the network admits multistationarity.

It is straightforward to check that $G$ has the embedded one-species networks
  $$\begin{aligned}
      X_1&\rightarrow 0\\
      2X_1&\rightarrow 3X_1
    \end{aligned}  ~~~~ \text{and}~~~~
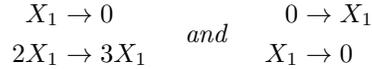
    \begin{aligned}
           0&\rightarrow X_1\\
            X_1&\rightarrow 0
    \end{aligned}$$ with arrow diagrams $(\leftarrow,\rightarrow)$ and $(\rightarrow,\leftarrow)$ respectively. So, this example verifies Theorem \ref{thm:Q6.1}.
\end{example}

We remark that Theorem  \ref{thm:Q6.1} (or Corollary \ref{cry:inverse}) might not be true if $G$ admits infinitely many degenerate positive steady states,
see the following example.

  \begin{example}\label{ex}
    Consider the following network:
   $$G:\begin{aligned}
      X_1&\rightarrow 2X_1\\
      X_1&\rightarrow 0
    \end{aligned}$$
    The arrow diagram of $G$ is $(	\leftarrow\!\!\!\!\bullet\!\!\!\!\rightarrow )$, so by \cite[Theorem 3.6 2(a)]{Joshi:Shiu:Multistationary}, $G$ admits infinitely many degenerate positive steady states. Therefore, $G$ admits multistationarity but does not have any embedded one-species network with arrow diagram $(\leftarrow, \rightarrow)$ or $(\rightarrow, \leftarrow)$.
  \end{example}

We also remark that the inverse of Theorem \ref{thm:Q6.1} does not hold (see Example \ref{ex:counter}) unless we put some additional condition (e.g., see Theorem \ref{thm:inverse}).

\begin{example}\label{ex:counter}

  Consider the following network:
  $$G:\begin{aligned}
      2X_1+X_2&\rightarrow X_1\\
      X_1+2X_2&\rightarrow 2X_1+3X_2\\
      X_1+X_2&\rightarrow 0
    \end{aligned}$$

It is straightforward to check that $G$ has the embedded one-species networks
  $$\begin{aligned}
      X_1&\rightarrow 2X_1\\
         2X_1&\rightarrow X_1
    \end{aligned}  ~~~~ \text{and}~~~~
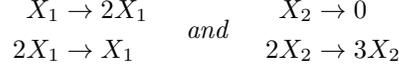
    \begin{aligned}
      X_2&\rightarrow 0\\
      2X_2&\rightarrow 3X_2
    \end{aligned}$$ with arrow diagrams  $(\rightarrow,\leftarrow)$ and $(\leftarrow,\rightarrow)$ respectively.

     For this network, we have
     $h_1(x)=x_1x_2(-\kappa_1x_1+\kappa_2x_2-\kappa_3)$, and $h_2(x)=-x_2+x_1-c_1$.  So,  the system $h_1(x)=h_2(x)=0$ has at most one positive solution for any positive $\kappa$ and for any real $c_1$. Hence, $G$ admits no multistationarity.
\end{example}


\begin{theorem}\label{thm:inverse}
Given a network $G$ \eqref{eq:network} with a one-dimensional stoichiometric subspace, 
if
\begin{itemize}
\item[(1)]a subnetwork consisting of two reactions in $G$ admits at least one and finitely many positive steady states, and
\item[(2)]$G$ has an embedded one-species network with arrow diagram $(\leftarrow, \rightarrow)$ and another with arrow diagram $(\rightarrow, \leftarrow)$,
\end{itemize}
then  the network $G$ admits multistationarity.
\end{theorem}

The following corollary directly follows from Theorem \ref{thm:Q6.1}, Theorem \ref{thm:inverse} and \cite[Lemma 4.1]{Joshi:Shiu:Multistationary}, which is consistent with
\cite[Theorem 5.1]{Joshi:Shiu:Multistationary}.

\begin{corollary}\label{cry:inverse}
  Given a network $G$ \eqref{eq:network} with two reactions,  assume $cap_{pos}(G)$ $<+\infty$.
The network $G$ admits multistationarity if and only if
\begin{itemize}
\item[(1)] there exists $\lambda<0$ such that $\beta_{k1}-\alpha_{k1}=\lambda(\beta_{k2}-\alpha_{k2})$ for any $k\in \{1, \ldots, s\}$, and
\item[(2)] $G$ has an embedded one-species network with arrow diagram $(\leftarrow, \rightarrow)$ and another with arrow diagram $(\rightarrow, \leftarrow)$.
    \end{itemize}
\end{corollary}

\subsection{Proof  of Theorem \ref{thm:Q6.1}}\label{sec:multip}
\begin{assumption}\label{assumption2}
  For the scalars $\lambda_j$ defined in \eqref{eq:scalar},
  we assume that $\lambda_i>0$ for all $i\in\{1, \ldots,  t\}$ and $\lambda_i<0$ for all $i\in \{t+1, \ldots, m\}$, where $t\in \{1,\dots,m\}$.
  Note that $\lambda_1>0$. So, $t\geq 1$.
  Also, note that if $t=m$, then  we have $\lambda_i>0$ for all $i\in \{1, \ldots, m\}$.
\end{assumption}

\begin{lemma}\label{lm:arrow}
  Given a network $G$ \eqref{eq:network} with a one-dimensional stoichiometric subspace,  we have the following statements.
  \begin{itemize}
    \item[(a)] $G$ has an embedded one-species network with arrow diagram $(\rightarrow,\leftarrow)$ if and only if
    there exist $k\in \{1,\ldots, s\}$, $j\in \{1, \ldots, t\}$ and $\ell \in \{t+1, \ldots, m\}$ such that
    $(\alpha_{kj}-\alpha_{k\ell})(\beta_{kj}-\alpha_{kj})<0$.
    \item[(b)]  $G$ has an embedded one-species network with arrow diagram $(\leftarrow,\rightarrow)$ if and only if  there exist $k\in \{1,\ldots, s\}$, $j\in \{1, \ldots, t\}$ and $\ell \in \{t+1, \ldots, m\}$ such that
    $(\alpha_{kj}-\alpha_{k\ell})(\beta_{kj}-\alpha_{kj})>0$.

  \end{itemize}
\end{lemma}

\begin{proof}
We only prove the part (a). The part (b) can be similarly proved.

$\Rightarrow$: If $G$ has an embedded one-species network with arrow diagram $(\rightarrow,\leftarrow)$, then we can assume this embedded one-species network is $$\begin{aligned}
  \alpha_{kj}X_k \rightarrow \beta_{kj}X_k\\
  \alpha_{k\ell}X_k \rightarrow \beta_{k\ell}X_k
\end{aligned}~,$$
where $k\in \{1,\dots,s\}$. Without loss of generality, we  assume $\lambda_j>0$, i.e., $j\in \{1,\dots,t\}$. By Definition \ref{def:arrow}, we know that $sign(\beta_{kj}-\alpha_{kj})=-sign(\beta_{k\ell}-\alpha_{k\ell})$. Thus, by \eqref{eq:scalar},  $\ell\in\{t+1,\dots,m\}$. More than that, by Definition \ref{def:arrow}, if $\alpha_{kj}>\alpha_{k\ell}$, then $\beta_{kj}-\alpha_{kj}<0$, and if $\alpha_{kj}<\alpha_{k\ell}$, then $\beta_{kj}-\alpha_{kj}>0$. Above all, we have $(\alpha_{kj}-\alpha_{k\ell})(\beta_{kj}-\alpha_{kj})<0$.

$\Leftarrow$:
Suppose there exist $k\in \{1,\ldots, s\}$, $j\in \{1, \ldots, t\}$ and $\ell \in \{t+1, \ldots, m\}$ such that
    $(\alpha_{kj}-\alpha_{k\ell})(\beta_{kj}-\alpha_{kj})<0$.
Consider the following embedded one-species network:
$$\begin{aligned}
  \alpha_{kj}X_k \rightarrow \beta_{kj}X_k\\
  \alpha_{k\ell}X_k \rightarrow \beta_{k\ell}X_k
\end{aligned}.$$
If $\alpha_{kj}>\alpha_{k\ell}$, then by the hypothesis $(\alpha_{kj}-\alpha_{k\ell})(\beta_{kj}-\alpha_{kj})<0$, we have $\beta_{kj}-\alpha_{kj}<0$.
Note that $\lambda_j\lambda_{\ell}<0$. Hence, by \eqref{eq:scalar}, we have $\beta_{k\ell}-\alpha_{k\ell}<0$.
So the arrow diagram of the network is $(\rightarrow,\leftarrow)$. Similarly, if $\alpha_{kj}<\alpha_{k\ell}$, then we can show that the arrow diagram of the above network is also $(\rightarrow,\leftarrow)$.
\end{proof}

\begin{lemma}\label{lm:t<m}
 Given a network  $G$ with a one-dimensional stoichiometric subspace,   if $cap_{pos}(G)>0$, then there exists $j\in \{1, \ldots, m\}$ such that $\lambda_j<0$, where $\lambda_j$ is the scalar defined in \eqref{eq:scalar}.
\end{lemma}

\begin{proof}
  If $\lambda_1>0,\dots, \lambda_m>0$, then by \eqref{eq:h1} , $h_1(x)=0$ has no positive solutions for $x$, which is a contradiction to  $cap_{pos}(G)>0$.
\end{proof}

{\bf Proof of Theorem \ref{thm:Q6.1}.}
\begin{proof}

Suppose that for some $\kappa\in {\mathbb R}_{>0}^m$, $G$ has two distinct positive steady states $x^{(1)}$, $x^{(2)}$  in a stoichiometric compatibility class ${\mathcal P}_c$. Without loss of generality, we assume the first coordinate of $x^{(1)}$ is larger than that of
$x^{(2)}$, i.e.,  $x_1^{(1)}>x_1^{(2)}$. Note that for the polynomial $h_1$ in \eqref{eq:h1},  we have $h_1(x^{(1)})=h_1(x^{(2)})=0$, i.e.,
\begin{align}\label{eq:h1i}
h_1(x^{(i)}) \;=\;\left(\beta_{11}-\alpha_{11}\right) \left(\sum\limits_{j=1}^m\lambda_j\kappa_j\prod\limits_{k=1}^s \left(x_k^{(i)}\right)^{\alpha_{kj}}\right)=0, \;\; i=1,2.
\end{align}
For $j=1, \ldots, m$, we define
$$ A_j^{(i)}\;:=\; \left\{
\begin{aligned}
 \lambda_j\kappa_j\prod\limits_{k=1}^s \left(x_k^{(i)}\right)^{\alpha_{kj}},&~~ \text{if}~~ j\in \{1,\dots, t\}\;(\text{i.e.},\; \lambda_j>0),  \\
 -\lambda_j\kappa_j\prod\limits_{k=1}^s \left(x_k^{(i)}\right)^{\alpha_{kj}},&~~ \text{if}~~ j\in \{t+1,\dots, m\} \;(\text{i.e.},\; \lambda_j<0).
\end{aligned}\right.
$$
By \eqref{eq:h1i}, we have
$$\sum\limits_{j=1}^tA_j^{(1)}=\sum\limits_{j=t+1}^mA_j^{(1)}~,\;\;\; \text{and}\;~
\sum\limits_{j=1}^tA_j^{(2)}=\sum\limits_{j=t+1}^mA_j^{(2)}.$$
Note that $A_j^{(i)}\neq 0$ (since $x^{(i)}\in {\mathbb R}^s_{>0}$, $\kappa_j>0$ and $\lambda_j\neq 0$). Recall Assumption \ref{assumption2} that we have $t\geq 1$, and so
$\sum\limits_{j=1}^tA_j^{(i)}\neq 0$ for $i=1, 2$.
 Note also, by Lemma \ref{lm:t<m}, if $cap_{pos}(G)>0$, then
$\{t+1, \ldots, m\}\neq \emptyset$, and so $\sum\limits_{j=t+1}^mA_j^{(i)}\neq 0$ for $i=1, 2$.
So, we have \begin{align}\label{eq:proportion}
  \frac{\sum\limits_{j=1}^tA_j^{(1)}}{\sum\limits_{j=1}^tA_j^{(2)}}
  ~=~\frac{\sum\limits_{j=t+1}^mA_j^{(1)}}{\sum\limits_{j=t+1}^mA_j^{(2)}}.
\end{align}

 Below, we prove by contradiction.
 Without loss of generality, we assume that  $G$ does not have any embedded one-species network with the arrow diagram
 $(\leftarrow,\rightarrow)$. Then,  by Lemma  \ref{lm:arrow} (b), for any $k\in \{1,\dots,s\}$,  for any $j\in \{1,\dots, t\}$ and for any $\ell\in \{t+1,\dots, m\}$, we have
 \begin{align}\label{eq:arrow}
 (\alpha_{kj}-\alpha_{k\ell})(\beta_{kj}-\alpha_{kj})\leq 0.
 \end{align}
 Note that $\lambda_j>0$. So, by \eqref{eq:scalar}, we know that $\beta_{kj}-\alpha_{kj}$ have the same sign with
 $\beta_{k1}-\alpha_{k1}$. Therefore,  for any $k\in \{1,\dots,s\}$, we have the following results.
 \begin{itemize}
  \item [(a)] If  $\beta_{kj}-\alpha_{kj}=0$, then $\beta_{k1}-\alpha_{k1}=0$. So,
  by $h_k(x^{(1)})=h_k(x^{(2)})=0$ (here, recall that $h_k$ is defined in \eqref{eq:h}), we have
  $$\frac{x_k^{(1)}}{x_k^{(2)}}=
\frac{
\frac{\beta_{k1}-\alpha_{k1}}{\beta_{11}-\alpha_{11}}x_1^{(1)}
 -\frac{c_{k-1}}{\beta_{11}-\alpha_{11}}}
 {\frac{\beta_{k1}-\alpha_{k1}}
 {\beta_{11}-\alpha_{11}}x_1^{(2)}-\frac{c_{k-1}}{\beta_{11}-\alpha_{11}}
 }=1
 ,$$
 and hence, for any $j\in \{1,\dots, t\}$ and for any $\ell\in \{t+1,\dots, m\}$,
 \begin{align}\label{eq:leqa}
\left( \frac{x_k^{(1)}}{x_k^{(2)}}\right)^{\alpha_{kj}}
 =\left( \frac{x_k^{(1)}}{x_k^{(2)}}\right)^{\alpha_{k\ell}}.
 \end{align}
  \item [(b)]If  $(\beta_{kj}-\alpha_{kj})>0$, then $\beta_{k1}-\alpha_{k1}>0$, and by \eqref{eq:arrow}, we have
  $\alpha_{kj}\leq \alpha_{k\ell}$ for any $j\in \{1,\dots, t\}$ and for any $\ell\in \{t+1,\dots, m\}$. Recall $x^{(1)}_1>x^{(2)}_1>0$. So, we have
   {\tiny
   \begin{align}\label{eq:leqb}
\left( \frac{x_k^{(1)}}{x_k^{(2)}}\right)^{\alpha_{kj}}=\left(\frac{
(\beta_{k1}-\alpha_{k1})x_1^{(1)}
-c_{k-1}}
 {(\beta_{k1}-\alpha_{k1})
 x_1^{(2)}-c_{k-1}
 }\right)^{\alpha_{kj}}
 \leq \left(
\frac{(\beta_{k1}-\alpha_{k1})x_1^{(1)}
 -c_{k-1}}
 {(\beta_{k1}-\alpha_{k1})x_1^{(2)}-c_{k-1}
 }\right)^{\alpha_{k\ell}}= \left( \frac{x_k^{(1)}}{x_k^{(2)}}\right)^{\alpha_{k\ell}}.
 \end{align}
 }
 \item [(c)]If  $(\beta_{kj}-\alpha_{kj})<0$, then $\beta_{k1}-\alpha_{k1}<0$, and by \eqref{eq:arrow}, we have
    $\alpha_{kj}\geq \alpha_{k\ell}$ for any $j\in \{1,\dots, t\}$ and for any $l\in \{t+1,\dots, m\}$. So, we still have
  {\tiny
   \begin{align}\label{eq:leqc}
\left( \frac{x_k^{(1)}}{x_k^{(2)}}\right)^{\alpha_{kj}}=\left(\frac{
(\beta_{k1}-\alpha_{k1})x_1^{(1)}
 -c_{k-1}}
 {(\beta_{k1}-\alpha_{k1})
 x_1^{(2)}-c_{k-1}
 }\right)^{\alpha_{kj}}
 \leq \left(
\frac{(\beta_{k1}-\alpha_{k1})x_1^{(1)}
 -c_{k-1}}
 {(\beta_{k1}-\alpha_{k1})x_1^{(2)}-c_{k-1}
 }\right)^{\alpha_{k\ell}}= \left( \frac{x_k^{(1)}}{x_k^{(2)}}\right)^{\alpha_{k\ell}}.
 \end{align}
 }\end{itemize}
 So, by (a), (b) and (c),  for any $j\in \{1,\dots, t\}$ and for any $\ell\in \{t+1,\dots, m\}$, we have
\begin{align}\label{eq:leq1}
 \frac{A_j^{(1)}}{A_j^{(2)}}
 = \prod\limits_{k=1}^s\left(\frac{x_k^{(1)}}{x_k^{(2)}}\right)^{\alpha_{kj}}
 \leq \prod\limits_{k=1}^s\left(\frac{x_k^{(1)}}{x_k^{(2)}}\right)^{\alpha_{k\ell}}
 =\frac{A_{\ell}^{(1)}}{A_{\ell}^{(2)}},
\end{align}
and hence,
 \begin{align}\label{eq:leq}
  \frac{\sum\limits_{j=1}^tA_j^{(1)}}{\sum\limits_{j=1}^tA_j^{(2)}}
  \leq  \max\limits_{1\leq j\leq t} \frac{A_j^{(1)}}{A_j^{(2)}}\leq
  \min\limits_{t+1\leq \ell \leq m}   \frac{A_{\ell}^{(1)}}{A_{\ell}^{(2)}}
  \leq\frac{\sum\limits_{\ell=t+1}^mA_{\ell}^{(1)}}{\sum\limits_{\ell=t+1}^mA_{\ell}^{(2)}}.
\end{align}
So by \eqref{eq:proportion}, all the non-strict inequalities in \eqref{eq:leq} should be equalities.
Hence, the non-strict inequality in \eqref{eq:leq1} is an equality, and
the non-strict inequalities in \eqref{eq:leqb} and \eqref{eq:leqc} are also equalities.
By (b) and (c),  for any $k \in \{1,\dots,s\}$ such that $\beta_{k1}-\alpha_{k1}\neq 0$, we have $\alpha_{kj}=\alpha_{k\ell}$ for any $j\in \{1,\dots, t\}$ and for any $\ell\in \{t+1,\dots, m\}$, and hence  we have $\alpha_{kj}$'s are all the same for every
$j\in \{1, \ldots, m\}$.  Then, the equality \eqref{eq:h1i} can be rewritten as
\begin{align}\label{eq:h1i'} \left(\beta_{11}-\alpha_{11}\right)\left(\sum\limits_{j=1}^m\lambda_j\kappa_j\prod\limits_{k\;\text{s.t.}\;\beta_{k1}= \alpha_{k1}} \left(\frac{c_{k-1}}{\beta_{11}-\alpha_{11}}\right)^{\alpha_{kj}}\right)\prod\limits_{k\;\text{s.t.}\;\beta_{k1}\neq \alpha_{k1}} \left(x_k^{(i)}\right)^{\alpha_{k1}}=0. 
\end{align}
Since  $\left(\beta_{11}-\alpha_{11}\right)\neq 0$ and $x^{(i)}$ is  a positive steady state, we must have $$\sum\limits_{j=1}^m\lambda_j\kappa_j\prod\limits_{k\;\text{s.t.}\;\beta_{k1}= \alpha_{k1}} \left(\frac{c_{k-1}}{\beta_{11}-\alpha_{11}}\right)^{\alpha_{kj}}=0.$$
Thus, $h_1(x)$ is equal to
\begin{align}
\left(\beta_{11}-\alpha_{11}\right)\left(\sum\limits_{j=1}^m\lambda_j\kappa_j\prod\limits_{k\;\text{s.t.}\;\beta_{k1}= \alpha_{k1}} \left(\frac{c_{k-1}}{\beta_{11}-\alpha_{11}}\right)^{\alpha_{kj}}\right)\prod\limits_{k\;\text{s.t.}\;\beta_{k1}\neq \alpha_{k1}} x_k^{\alpha_{k1}}\equiv0.
\end{align}
 Therefore, the network $G$ has infinitely  many positive steady states in ${\mathcal P}_c$, which is a contradiction to the hypothesis that
 $cap_{pos}(G)<+\infty$.
\end{proof}
\subsection{Proof  of Theorem
\ref{thm:inverse}}\label{sec:multip2}
\begin{definition}
  \cite[Difinition 2.1]{Joshi:Shiu:Multistationary} A \defword{subnetwork} of a network $G$ is a network that consists of some reactions from the network $G$.
\end{definition}
\begin{lemma}\label{lm:lift}\cite[Lemma 2.12]{Joshi:Shiu:Multistationary}
Let $G^*$ be a subnetwork of a network $G$ which has the same stoichiometric subspace as $G$. Then $cap_{nondeg}(G^*)\leq cap_{nondeg}(G)$.
\end{lemma}

\begin{lemma}\label{lm:submss} Given a network  $G$ with a one-dimensional stoichiometric subspace,  
if there exist $i\in \{1,\dots,t\}$ and $j\in\{t+1,\dots,m\}$ such that
\begin{itemize}
\item[(1)]the subnetwork consisting of the $i$-th  and the $j$-th reactions of $G$ admits finitely many positive steady states, and
\item[(2)] there exist $k_1,k_2\in \{1,\dots,s\}$ ($k_1\neq k_2$) such that
$$(\alpha_{k_1i}-\alpha_{k_1j})(\beta_{k_1i}-\alpha_{k_1i})\cdot
(\alpha_{k_2i}-\alpha_{k_2j})(\beta_{k_2i}-\alpha_{k_2i})<0,$$
\end{itemize}
then $G$ admits multistationarity.
\end{lemma}
\begin{proof}
  Consider the bi-reaction subnetwork $G_1$ consisting of the $i$-th and the $j$-th reactions of $G$.
  By the hypothesis (2), the network $G_1$
  satisfies the sufficient condition of \cite[Theorem 5.1]{Joshi:Shiu:Multistationary}, and thus, $G_1$ admits multistationarity, i.e., $cap_{pos}(G_1)\geq 2$.
   By the hypothesis (1), we have $cap_{pos}(G_1)<+\infty$, and hence,  we deduce from Theorem \ref{thm:nc} that $cap_{nondeg}(G_1)=cap_{pos}(G_1)\geq 2$. Since Lemma \ref{lm:lift} shows that $cap_{nondeg}(G_1)\leq cap_{nondeg}(G)$, we have $cap_{pos}(G)\geq cap_{nondeg}(G) \geq 2$.
\end{proof}

\begin{lemma}\label{lm:nondegmss}\cite[Theorem 5.2]{Joshi:Shiu:Multistationary}
Let $G$ be a network \eqref{eq:network}  with two reactions.  Then $G$ is nondegenerately multistationary if and only if
\begin{itemize}
\item[(1)] there exists $\lambda<0$ such that $\beta_{k1}-\alpha_{k1}=\lambda(\beta_{k2}-\alpha_{k2})$ for any $k\in \{1, \ldots, s\}$,
  \item [(2)] in the set $\{(\alpha_{k1}-\alpha_{k2})(\beta_{k1}-\alpha_{k1})\}_{k=1}^s$ there are at least one positive and at least one negative number, and
  \item [(3)] if in the set $\{(\alpha_{k1}-\alpha_{k2})(\beta_{k1}-\alpha_{k1})\}_{k=1}^s$ there are exactly one positive and one negative number, say $(\alpha_{i1}-\alpha_{i2})(\beta_{i1}-\alpha_{i1})>0$ and $(\alpha_{j1}-\alpha_{j2})(\beta_{j1}-\alpha_{j1})<0$, then $\alpha_{i1}-\alpha_{i2}\neq -(\alpha_{j1}-\alpha_{j2})$ (here, $i,j\in \{1, \ldots,s\}$).
\end{itemize}
\end{lemma}

{\bf Proof of Theorem \ref{thm:inverse}.}

\begin{proof}
By the hypothesis (1), there exist two indices $i,j\in \{1, \ldots, m\}$ such that
for the subnetwork consisting of $i$-th and $j$-th reactions, say $G^*$, we have
$0<cap_{pos}(G^*)<+\infty$.
Note here, if $i, j\in \{1, \ldots, t\}$ or  $i, j\in \{t+1, \ldots, m\}$, then by a similar argument with the proof of Lemma \ref{lm:t<m}, we will have $cap_{pos}(G^*)=0$.  So, since $cap_{pos}(G^*)>0$, we can assume that
$i\in \{1, \ldots, t\}$ and $j\in \{t+1, \ldots, m\}$. Also notice that
if  $(\alpha_{ki}-\alpha_{kj})(\beta_{ki}-\alpha_{ki})=0$ for any $k\in \{1, \ldots, s\}$, then by \eqref{eq:h1}--\eqref{eq:h}, we will see that
$cap_{pos}(G^*)=+\infty$. So, since $cap_{pos}(G^*)<+\infty$, there exists  $k_1\in \{1, \ldots, s\}$ such that
$(\alpha_{k_1i}-\alpha_{k_1j})(\beta_{k_1i}-\alpha_{k_1i})\neq 0$.  Without loss of generality,
we assume
\begin{align}\label{eq:k1}
(\alpha_{k_1i}-\alpha_{k_1j})(\beta_{k_1i}-\alpha_{k_1i})<0.
\end{align}
Below, we discuss two cases.

{\bf (Case 1).} If there exists $k_2\in \{1, \ldots, s\}$ such that
$$
(\alpha_{k_2i}-\alpha_{k_2j})(\beta_{k_2i}-\alpha_{k_2i})>0,$$
then by Lemma \ref{lm:submss}, we have $cap_{pos}(G)\geq 2$ and we are done.

{\bf (Case 2).} Suppose for all $k\in \{1, \ldots, s\}$, $(\alpha_{ki}-\alpha_{kj})(\beta_{ki}-\alpha_{ki})\leq 0$.
By \eqref{eq:k1} and Lemma \ref{lm:arrow} (a), $G$ has an embedded  one-species network with arrow diagram $(\rightarrow, \leftarrow)$.
By the hypothesis (2), $G$ has another embedded  one-species network with arrow diagram $(\leftarrow, \rightarrow)$.
So, by Lemma \ref{lm:arrow} (b), there exist $i_2\in \{1, \ldots, t\}$, $j_2 \in \{t+1, \ldots, m\}$ and $k_3\in \{1, \ldots, s\}$ such that
\begin{align}\label{eq:k3}
(\alpha_{k_3i_2}-\alpha_{k_3j_2})(\beta_{k_3i_2}-\alpha_{k_3i_2})>0.
\end{align}
Consider the subnetwork consisting of the $i_2$-th and the $j_2$-th reactions, say $G^*_2$.
If $cap_{nondeg}(G^*_2)\geq 2$, then by Lemma \ref{lm:lift}, we have
$$cap_{pos}(G)\geq cap_{nondeg}(G^*_2)\geq 2, $$ and we are done.
So, in the rest of the proof, we  assume that $cap_{nondeg}(G^*_2)<2$.

Without loss of generality, we assume that
 $\alpha_{k_3i_2}-\alpha_{k_3j_2}>0$.
Pick any two positive numbers $y_1$ and $z_1$ such that $y_1>z_1>0$. For $k=2, \ldots, s$, we define
\begin{align}\label{eq:yz}
 y_k\;:=\;\frac{\beta_{k1}-\alpha_{k1}}{\beta_{11}-\alpha_{11}}y_1
 +\frac{c_{k-1}}{\beta_{11}-\alpha_{11}} \;\; \text{and} \;\; z_k\;:=\;\frac{\beta_{k1}-\alpha_{k1}}{\beta_{11}-\alpha_{11}}z_1
 +\frac{c_{k-1}}{\beta_{11}-\alpha_{11}},
 \end{align}
 where $c_1, \ldots c_{s-1}$ are chosen to be  real numbers satisfying the following two conditions.
 \begin{itemize}
 \item[(C1)] 
 For for every $k\in \{2, \ldots, s\}$,
 $y_k>0$ and $z_k>0$.
 \item[(C2)] For any $k$ such that $(\beta_{ki_2}-\alpha_{ki_2})(\beta_{k_3i_2}-\alpha_{k_3i_2})>0$, we have
  $$\left\{\begin{aligned}
   & \frac{c_{k-1}}{\beta_{k 1}-\alpha_{k1}}>\frac{c_{k_3-1}}{\beta_{k_3 1}-\alpha_{k_31}}>0,& \text{if} ~ \beta_{11}-\alpha_{11}>0\\
   &\frac{c_{k_3-1}}{\beta_{k_3 1}-\alpha_{k_31}}<\frac{c_{k-1}}{\beta_{k 1}-\alpha_{k1}}<-y_1,& \text{if} ~ \beta_{11}-\alpha_{11}<0
 \end{aligned}\right.$$
 \end{itemize}
 It is straightforward to check that (C2) is consistent with (C1).
 Now we define the following two continuous functions in $\kappa$:
\begin{align}\label{eq:varphi}
\varphi(\kappa_1,\dots,\kappa_t):=\frac
{\sum\limits_{i=1}^t\lambda_i\kappa_i\prod\limits_{k=1}^s y_k^{\alpha_{ki}}}{\sum\limits_{i=1}^t\lambda_i\kappa_i\prod\limits_{k=1}^s {z_k}^{\alpha_{ki}}}~~,~~ \psi(\kappa_{t+1},\dots,\kappa_m):=\frac
{\sum\limits_{j=t+1}^m\lambda_j\kappa_j\prod\limits_{k=1}^s y_k^{\alpha_{kj}}}{\sum\limits_{j=t+1}^m\lambda_j\kappa_j\prod\limits_{k=1}^s {z_k}^{\alpha_{kj}}}.
\end{align}

{\bf Claim 1.} If for all $k\in \{1, \ldots, s\}$, we have $(\alpha_{ki}-\alpha_{kj})(\beta_{ki}-\alpha_{ki})\leq 0$, then
there exits $\kappa^{(1)}=(\kappa_1^{(1)}, \ldots, \kappa_m^{(1)})\in {\mathbb R}^m_{>0}$ such that
\begin{align}\label{eq:min<max}
 \varphi(\kappa^{(1)}_1,\dots,\kappa^{(1)}_t)<\psi(\kappa^{(1)}_{t+1},\dots,\kappa^{(1)}_m).
\end{align}

{\bf Claim 2.} If $cap_{nondeg}(G^*_2)<2$, then
there exists $\kappa^{(2)}=(\kappa_1^{(2)}, \ldots, \kappa_m^{(2)})\in {\mathbb R}^m_{>0}$ such that
\begin{align}\label{eq:max>min}
\varphi(\kappa^{(2)}_1,\dots,\kappa^{(2)}_t)>\psi(\kappa^{(2)}_{t+1},\dots,\kappa^{(2)}_m)
\end{align}
(we will prove the claims  later).

If both Claim 1 and Claim 2 are true, then
 by \eqref{eq:min<max} and \eqref{eq:max>min}, there exists $\kappa^*=(\kappa_1^*,\dots,\kappa_m^*)\in {\mathbb R}^m_{>0}$ such that
 \begin{align}\label{eq:min=max}
 \varphi(\kappa_1^*,\dots,\kappa_t^*)=\psi(\kappa_{t+1}^*,\dots,\kappa_m^*).
 \end{align}
 Now, we define two positive numbers $$B:=\sum\limits_{i=1}^t\lambda_i\kappa_i^*\prod\limits_{k=1}^s z_k^{\alpha_{ki}}~, ~ C:=-\sum\limits_{j=t+1}^m\lambda_j\kappa_j^*\prod\limits_{k=1}^s z_k^{\alpha_{kj}}.$$
 For $i\in\{1,\dots,t\}$ and $j\in \{t+1,\dots, m\}$,
 let $\kappa_i^{**}:=\kappa_i^*/B>0$, $\kappa_j^{**}:=\kappa_j^*/C>0$. Then, it is straightforward to check that
 $$\sum\limits_{i=1}^t\lambda_i\kappa_i^{**}\prod\limits_{k=1}^s {z_k}^{\alpha_{ki}}=1,\;
 \text{and}\; \sum\limits_{j=t+1}^m\lambda_j\kappa_j^{**}\prod\limits_{k=1}^s {z_k}^{\alpha_{kj}}=-1.$$
 So, we have
 $$\sum\limits_{i=1}^t\lambda_i\kappa_i^{**}\prod\limits_{k=1}^s {z_k}^{\alpha_{ki}}+\sum\limits_{j=t+1}^m\lambda_j\kappa_j^{**}\prod\limits_{k=1}^s {z_k}^{\alpha_{kj}}=0.$$
 Also, by \eqref{eq:varphi} and  \eqref{eq:min=max}, we can check that
 $$\sum\limits_{i=1}^t\lambda_i\kappa_i^{**}\prod\limits_{k=1}^s y_k^{\alpha_{ki}}+\sum\limits_{j=t+1}^m\lambda_j\kappa_j^{**}\prod\limits_{k=1}^s y_k^{\alpha_{kj}}=0.$$
Therefore, for the rate-constant vector $\kappa^{**}=(\kappa_1^{**},\dots, \kappa_m^{**})$, $G$ has at least two positive steady states $y=(y_1, \ldots, y_s)$ and $z=(z_1, \ldots, z_s)$ in ${\mathcal P}_{c^*}$, and we are done.

{\bf Proof of Claim 1.} Recall by Assumption \ref{assumption2} that $\lambda_{i}>0$ since $i\in \{1, \ldots, t\}$. So, by \eqref{eq:scalar}, $\beta_{ki}-\alpha_{ki}$ and $\beta_{k1}-\alpha_{k1}$ have the same sign. So, for all $k\in \{1,\dots,s\}$, if we have $(\alpha_{ki}-\alpha_{kj})(\beta_{ki}-\alpha_{ki})
\leq 0$, then
     $(\alpha_{ki}-\alpha_{kj})(\beta_{k1}-\alpha_{k1})\leq 0$.
      Therefore,
      {\tiny
 \begin{align}
   \left(\frac{y_k}{z_k}\right)^{\alpha_{ki}}=
   \left(\frac{
(\beta_{k1}-\alpha_{k1})y_1
 +c_{k-1}}
 {(\beta_{k1}-\alpha_{k1})
 z_1+c_{k-1}
 }\right)^{\alpha_{ki}}\leq \left(\frac{
(\beta_{k1}-\alpha_{k1})y_1
 +c_{k-1}}
 {(\beta_{k1}-\alpha_{k1})
 z_1+c_{k-1}
 }\right)^{\alpha_{kj}}=\left(\frac{y_k}{z_k}\right)^{\alpha_{kj}}.
 \end{align}
 }Note that by \eqref{eq:k1}, for $k=k1$, the above inequality is strict.
Hence,
for the rate constants $\kappa_{i}=1$, $\kappa_{j}=1$ , and $\kappa_{\ell}=0$ ($\ell\in
\{1, \ldots, m\}\backslash\{i, j\}$),
\begin{align}
\varphi(\kappa_1,\dots,\kappa_t)= \prod\limits_{k=1}^{s}\left(\frac{y_k}{z_k}\right)^{\alpha_{ki}}
 <\prod\limits_{k=1}^{s}\left(\frac{y_k}{z_k}\right)^{\alpha_{kj}} =\psi(\kappa_{t+1},\dots,\kappa_m).
\end{align} So, there exits $\kappa^{(1)}=(\kappa_1^{(1)}, \ldots, \kappa_m^{(1)})\in {\mathbb R}^m_{>0}$ such that
\eqref{eq:min<max} holds
(here, for instance, we can choose $\kappa^{(1)}_{i}=1$, $\kappa^{(1)}_{j}=1$ , and $\kappa^{(1)}_{\ell}=\epsilon$ for $\ell\in
\{1, \ldots, m\}\backslash\{i, j\}$, where $\epsilon$ is a sufficiently small positive number).

{\bf Proof of Claim 2.} Recall that $G_2^*$ is a bi-reaction network consisting of the $i_2$-th and $j_2$-th reactions in $G$, where
$i_2\in \{1, \ldots, t\}$ and $j_2\in \{t+1, \ldots, m\}$. By \eqref{eq:scalar}, we know that for all $k\in \{1, \ldots, s\}$,
$(\beta_{ki_2}-\alpha_{ki_2})=\frac{\lambda_{i_2}}{\lambda_{j_2}}(\beta_{kj_2}-\alpha_{kj_2})$, and by Assumption \ref{assumption2},
$\frac{\lambda_{i_2}}{\lambda_{j_2}}<0$. So, the condition (1) stated in Lemma \ref{lm:nondegmss} always holds for $G_2^*$.
If $cap_{nondeg}(G_2^*)<2$, then by Lemma \ref{lm:nondegmss}, for the network $G_2^*$, either the condition (2) or the condition (3) in Lemma \ref{lm:nondegmss} does not hold.

Recall that for some $k_3\in \{1, \ldots, s\}$, we have \eqref{eq:k3}: $(\alpha_{k_3i_2}-\alpha_{k_3j_2})(\beta_{k_3i_2}-\alpha_{k_3i_2})>0$. So, if the condition (2) in Lemma \ref{lm:nondegmss} does not hold for $G_2^*$, then
for all $k\in \{1, \ldots, s\}$, we have $(\alpha_{ki_2}-\alpha_{kj_2})(\beta_{ki_2}-\alpha_{ki_2})\geq 0$, and by a similar argument with the proof of the claim 1, we can prove the claim 2.

 Suppose the condition (3) in Lemma \ref{lm:nondegmss} does not hold for $G_2^*$. We can assume that there exists $k_4\in \{1, \ldots,s\}$ such that $(\alpha_{k_4i_2}-\alpha_{k_4j_2})(\beta_{k_4i_2}-\alpha_{k_4i_2})<0$ and
 \begin{align}\label{eq:cond31}
 \text{for any}\; k\in \{1, \ldots, s\}\backslash \{k_3, k_4\},\; (\alpha_{ki_2}-\alpha_{kj_2})(\beta_{ki_2}-\alpha_{ki_2})=0.
 \end{align}
 Since the condition (3) in  Lemma \ref{lm:nondegmss} does not hold, we have
 \begin{align}\label{eq:cond32}
 \alpha_{k_3i_2}-\alpha_{k_3j_2}= -(\alpha_{k_4i_2}-\alpha_{k_4j_2}).
 \end{align}
 By the proof of the claim 1, we only need to show that
 \begin{align}\label{eq:need1}
 \prod\limits_{k=1}^{s}\left(\frac{y_k}{z_k}\right)^{\alpha_{ki_2}}
 >\prod\limits_{k=1}^{s}\left(\frac{y_k}{z_k}\right)^{\alpha_{kj_2}}.
\end{align}
By \eqref{eq:cond31},
for any $k\in \{1, \ldots, s\}\backslash \{k_3, k_4\}$, we have $\left(\frac{y_k}{z_k}\right)^{\alpha_{ki_2}}=\left(\frac{y_k}{z_k}\right)^{\alpha_{kj_2}}$.
So, the inequality \eqref{eq:need1} becomes
\begin{align}
 &\left(\frac{y_{k_3}}{z_{k_3}}\right)^{\alpha_{k_3i_2}}\left(\frac{y_{k_4}}{z_{k_4}}\right)^{\alpha_{k_4i_2}}
 >\left(\frac{y_{k_3}}{z_{k_3}}\right)^{\alpha_{k_3j_2}}\left(\frac{y_{k_4}}{z_{k_4}}\right)^{\alpha_{k_4j_2}}\notag\\
 \Leftrightarrow & \left(\frac{y_{k_3}}{z_{k_3}}\right)^{\alpha_{k_3i_2}-\alpha_{k_3j_2}}\left(\frac{y_{k_4}}{z_{k_4}}\right)^{\alpha_{k_4i_2}-\alpha_{k_4j_2}}>1\label{eq:need2}
 \end{align}
Recall that
 $\alpha_{k_3i_2}-\alpha_{k_3j_2}>0$.
 By \eqref{eq:cond32}, in order to prove \eqref{eq:need2}, we only need to show
 \begin{align}
 &\left(\frac{y_{k_3}}{z_{k_3}}\right)\left(\frac{y_{k_4}}{z_{k_4}}\right)^{-1}>1 \notag\\
&  \Leftrightarrow
   \frac{(\beta_{k_3 1}-\alpha_{k_3 1})(\beta_{k_4 1}-\alpha_{k_4 1})}{(\beta_{11}-\alpha_{11})^2}\left(\frac{c_{k_4-1}}{\beta_{k_4 1}-\alpha_{k_41}}-\frac{c_{k_3-1}}{\beta_{k_3 1}-\alpha_{k_31}}\right)\left(y_1-z_1\right)>0\label{eq:main}
 \end{align}
 (here, the last equivalence follows from \eqref{eq:yz}).
 Recall that $$(\alpha_{k_3i_2}-\alpha_{k_3j_2})(\beta_{k_3 i_2}-\alpha_{k_3 i_2})>0, ~\text{and}~ (\alpha_{k_4i_2}-\alpha_{k_4j_2})(\beta_{k_4 i_2}-\alpha_{k_4 i_2})<0.$$ So, we have
 \begin{align}\label{eq:C2cond}
 \beta_{k_3 i_2}-\alpha_{k_3 i_2}>0,~\text{and}~ \beta_{k_4 i_2}-\alpha_{k_4 i_2}>0.
 \end{align}
  Note also, by Assumption \ref{assumption2}, we have $\lambda_{i_2}>0$.
  So, by \eqref{eq:scalar}, we have \begin{align}
   \beta_{k_3 1}-\alpha_{k_3 1}=\lambda_{i_2}^{-1}(\beta_{k_3 i_2}-\alpha_{k_3 i_2})>0, ~\text{and} ~\beta_{k_4 1}-\alpha_{k_4 1}=\lambda_{i_2}^{-1}(\beta_{k_4 i_2}-\alpha_{k_4 i_2})>0.\label{eq:k3k4}
 \end{align}
 Since $y_1>z_1$, we deduce from \eqref{eq:k3k4} that $$\eqref{eq:main} \Leftrightarrow \frac{c_{k_4-1}}{\beta_{k_4 1}-\alpha_{k_41}}>\frac{c_{k_3-1}}{\beta_{k_3 1}-\alpha_{k_31}}.$$
Thus, by \eqref{eq:C2cond} and (C2),  we are done with the proof.
\end{proof}


\section{A necessary condition for admitting three positive steady states}\label{sec:nec}

In this section, we will extend Theorem \ref{thm:Q6.1} to the networks admitting at least three positive steady states. We first introduce a new concept ``bi-arrow diagram". Then, we point out that Theorem \ref{thm:Q6.1} shows that multistationarity implies two bi-arrow diagrams (Corollary \ref{cry:ad2}). The main result of this section is Theorem
\ref{thm:ad3}: admitting three positive steady states implies three bi-arrow diagrams.
 Recall that multistability implies three positive steady states. So, another remarkable corollary here is that multistability implies  three bi-arrow diagrams.

First, we recall that for any network \eqref{eq:network} with a one-dimensional stoichiometric subspace, we have non-zero
scalars $\lambda_i$ defined in \eqref{eq:scalar}.
By Assumption \ref{assumption2},    $t$ is the index in $\{1, \ldots, m\}$ such that  $\lambda_i>0$ for all $i\in\{1, \ldots,  t\}$ and $\lambda_i<0$ for all $i\in \{t+1, \ldots, m\}$.

\begin{definition}\label{def:ad}
(1) For a network $G$ with a one-dimensional stoichiometric subspace, if there exist $k\in \{1,\dots, s\},~ i \in \{1,\dots, t\}, $ and $j\in \{t+1,\dots,m\}$ such that  $(\alpha_{ki}-\alpha_{kj})(\beta_{ki}-\alpha_{ki})\neq0$, then we say $G$ has a \defword{bi-arrow diagram}.

(2) $Ad(G):=|\{(k, i, j)\in\{1,\dots, s\}\times \{1,\dots, t\} \times \{t+1,\dots,m\}|(\alpha_{ki}-\alpha_{kj})(\beta_{ki}-\alpha_{ki})\neq0\}|$, i.e., $Ad(G)$ denotes
the number of bi-arrow diagrams the network $G$ has.
\end{definition}

\begin{corollary}\label{cry:ad2}Given a network $G$ with a one-dimensional stoichiometric subspace,
if $2\leq cap_{pos}(G)<+\infty$, then $Ad(G)\geq 2$.
\end{corollary}
\begin{proof}
The conclusion directly follows from Theorem \ref{thm:Q6.1}, Lemma \ref{lm:arrow} and Definition \ref{def:ad} (2).
\end{proof}

\begin{theorem}\label{thm:ad3}
 Given a network $G$ with a one-dimensional stoichiometric subspace, if $3\leq cap_{pos}(G)<+\infty$, then $Ad(G)\geq 3$.
\end{theorem}

\begin{corollary}\label{cry:ad3}
Given a network $G$ with a one-dimensional stoichiometric subspace, if $3\leq cap_{pos}(G)<+\infty$,  then
one of the following two statements holds.
\begin{itemize}
 \item[(1)]$G$ has two embedded one-species networks with arrow diagram $(\leftarrow, \rightarrow)$ and one with arrow diagram $(\rightarrow, \leftarrow)$.
 \item[(2)]$G$ has two embedded one-species networks with arrow diagram  $(\rightarrow, \leftarrow)$ and one with arrow diagram $(\leftarrow, \rightarrow)$.
 \end{itemize}
\end{corollary}
\begin{proof}
By Theorem \ref{thm:Q6.1}, if $2\leq cap_{pos}(G)<+\infty$, then
$G$ has an embedded one-species networks with arrow diagram $(\leftarrow, \rightarrow)$ and another with arrow diagram $(\rightarrow, \leftarrow)$. By Lemma \ref{lm:arrow}, there exist $(i_1, j_1, k_1)$ and $(i_2, j_2, k_2)$ such that
$$(\alpha_{k_1i_1}-\alpha_{k_1j_1})(\beta_{k_1i_1}-\alpha_{k_1i_1})>0, \;\text{and}\;(\alpha_{k_2i_2}-\alpha_{k_2j_2})(\beta_{k_2i_2}-\alpha_{k_2i_2})<0.$$
By Theorem \ref{thm:ad3},  if $3\leq cap_{pos}(G)<+\infty$, then $Ad(G)\geq 3$. So, by Definition \ref{def:ad}, there exists $(i_3, j_3, k_3)$, which is different from $(i_1, j_1, k_1)$ or $(i_2, j_2, k_2)$ such that
$$(\alpha_{k_3i_3}-\alpha_{k_3j_3})(\beta_{k_3i_3}-\alpha_{k_3i_3})\neq 0.$$
By Lemma \ref{lm:arrow},  the triple $(i_3, j_3, k_3)$ corresponds to a third embedded one-species network with arrow diagram $(\leftarrow, \rightarrow)$ or $(\rightarrow, \leftarrow)$.
\end{proof}

\begin{corollary}\label{cry:ad3stable}
 Given a network $G$ with a one-dimensional stoichiometric subspace, if $G$ admits multistability, then $Ad(G)\geq 3$, and one of the statements (1)--(2) in Corollary \ref{cry:ad3} holds.
\end{corollary}

\begin{proof}
Since we have $cap_{stab}(G)\leq \frac{cap_{pos}(G)+1}{2}$ (see \cite[Theorem 3.4]{txzs2020}),
the result follows from Theorem
\ref{thm:ad3} and Corollary \ref{cry:ad3}. 

\end{proof}

\begin{example}
  This example illustrates how to count $Ad(G)$. Consider the following network:
  $$G:\begin{aligned}
    X_1+3X_2&\rightarrow 4X_2+X_3\\
    X_2+X_3&\rightarrow  X_1
  \end{aligned}$$
For this network, we have $s=3$, $m=2$ and $t=1$, and thus $\{1, \ldots, t\}=\{1\}$, $\{t+1, \ldots, m\}=\{2\}$.
  It is straightforward to verify that $(\alpha_{11}-\alpha_{12})(\beta_{11}-\alpha_{11})=-1\neq 0$, $(\alpha_{21}-\alpha_{22})(\beta_{21}-\alpha_{21})=2\neq 0$, and $(\alpha_{31}-\alpha_{32})(\beta_{31}-\alpha_{31})=-1\neq 0$. Hence, the network has $3$ bi-arrow diagrams, i.e., $Ad(G)=3$.

  On the other hand, by \cite[Theorem 2.4]{tx2020}, we know $cap_{pos}(G)\geq 3$. So this example also verifies Theorem \ref{thm:ad3}.
\end{example}

\subsection{Proof of Theorem \ref{thm:ad3}}
Throughout this section, we assume any network $G$ mentioned in the lemmas has a one-dimensional stoichiometric subspace. Given a network $G$ \eqref{eq:network}, we first define the following notions.
\begin{definition}\label{def:adnotion}
(1) $E:=\{k\in \{1,\dots, s\}| \alpha_{k1},\dots, \alpha_{km}\; \text{are not all the same}\}.$

(2) $H:=\{k\in\{1,\dots,s\}|\beta_{k1}- \alpha_{k1}\neq 0\}$.

(3) For $F\subset \{1,\dots, s\}$, $G_F$ denotes the embedded subnetwork of $G$ obtained by removing all the species except $\{X_k|k\in F\}$. Especially, we abbreviate $G_{\{k\}}$ as $G_k$.
\end{definition}

\begin{example}
  This example illustrates Definition \ref{def:adnotion}. Consider the following network
  $$G:\;\begin{aligned}
    X_1+3X_2+X_4+X_5&\rightarrow 4X_2+X_3+X_5\\
    X_2+X_3+X_4&\rightarrow  X_1+2X_4
  \end{aligned}$$
For this network, $s=5$ and $m=2$. By Definition \ref{def:adnotion}. $E=\{1,2,3,5\}$ and $H=\{1,2,3,4\}$. More than that, $G_H$ is obtained by removing $X_5$:
  $$G_H\;:\begin{aligned}
    X_1+3X_2+X_4&\rightarrow 4X_2+X_3\\
    X_2+X_3+X_4&\rightarrow  X_1+2X_4
  \end{aligned}$$
\end{example}

\begin{proposition}\label{prop:Ad}
  We have the following statements.

  \begin{itemize}
    \item[(1)]$Ad(G)=\sum\limits_{k=1}^sAd(G_k).$
    \item[(2)]For any $k\in H$, $Ad(G_k)=|\{(i, j)\in\{1,\dots, t\} \times \{t+1,\dots,m\}|\alpha_{ki}-\alpha_{kj}\neq0\}|$.
    \item[(3)]If $cap_{pos}(G)>0$, then for any $k\in \{1,\dots, s\}$, $Ad(G_k)\geq 1$ if and only if $k\in E\cap H$.
    \item[(4)]If $cap_{pos}(G)>0$, then $Ad(G)\geq |E\cap H|$.
    \item[(5)]For any $k\in H$, if the set $\{\alpha_{k1},\dots \alpha_{km} \}$ contains at least 3 distinct values, then $Ad(G_k)\geq 2$.
  \end{itemize}

\end{proposition}
\begin{proof}
(1) This conclusion  directly follows from Definition \ref{def:ad} (1)(2).

(2)
For any $k\in H$, by Definition \ref{def:adnotion} (2), $\beta_{k1}-\alpha_{k1}\neq 0$.
Note that for any
$i\in \{1, \ldots, m\}$, we have  $\beta_{ki}-\alpha_{ki}\neq0$ since $\beta_{ki}-\alpha_{ki}=\lambda_i(\beta_{k1}-\alpha_{k1})$ and $\lambda_i\neq 0$ (recall \eqref{eq:scalar}). So, the conclusion follows from Definition \ref{def:ad} (2).

(3)``$\Rightarrow$:"  If $Ad(G_k)\geq 1$, then by Definition \ref{def:ad} (2), there exist  $i \in \{1,\dots, t\}$ and $j\in \{t+1,\dots,m\}$ such that
$(\alpha_{ki}-\alpha_{kj})(\beta_{ki}-\alpha_{ki})\neq0$, i.e., $\alpha_{ki}-\alpha_{kj}\neq 0$ and $\beta_{ki}-\alpha_{ki}\neq0$. Note that by \eqref{eq:scalar}, we have $\beta_{k1}-\alpha_{k1}\neq0$ since $\beta_{ki}-\alpha_{ki}=\lambda_i(\beta_{k1}-\alpha_{k1})$ and $\lambda_i\neq 0$. So, by Definition \ref{def:adnotion} (1)(2),
we have $k\in E\cap H$.

``$\Leftarrow$:"
If $k\in E
\cap H$, then $k\in E$, and so by Definition \ref{def:adnotion} (1),  $\alpha_{k1},\dots \alpha_{km}$ are not all the same. By Lemma \ref{lm:t<m}, if $cap_{pos}(G)>0$, then  $\{t+1,\dots, m\}\neq \emptyset$. If there exists $j\in \{t+1,\dots, m\}$ such that $\alpha_{k1}\neq \alpha_{kj}$, then  by (2), $Ad(G_k)\geq 1$. If
$\alpha_{k1}= \alpha_{kj}$ for any $j\in \{t+1,\dots, m\}$, then there exists $j\in \{2,\dots, t\}$ such that $\alpha_{k1}\neq \alpha_{kj}$  since $\alpha_{k1},\dots \alpha_{km}$ are not all the same. So, we have $\alpha_{km}\neq \alpha_{kj}$, and thus  by (2), $Ad(G_k)\geq 1$.

(4) By (1) and (3), we have
$$Ad(G)=\sum\limits_{k=1}^sAd(G_k)\geq \sum\limits_{k\in E\cap H}Ad(G_k)\geq \sum\limits_{k\in E\cap H}1=|E\cap H|.$$

(5)Let $\alpha_{ki}, \alpha_{kj}, \alpha_{k\ell}$ be three different values. If $i,j,\ell\in \{1,\dots,t\}$, then at least two of the values $\alpha_{ki}-\alpha_{km},\alpha_{kj}-\alpha_{km},\alpha_{k\ell}-\alpha_{km}$ are non-zero. Since $k\in H$, by (2), we have $Ad(G_k)\geq 2$. Similarly, if $i,j,\ell\in \{t+1,\dots,m\}$, then at least two of $\alpha_{ki}-\alpha_{k1},\alpha_{kj}-\alpha_{k1},\alpha_{k\ell}-\alpha_{k1}$ are non-zero, and thus, $Ad(G_k)\geq 2$.
Without loss of generality, assume $i\in\{1,\dots,t\}$ and $j,\ell\in \{t+1,\dots,m\}$ such that $\alpha_{ki}, \alpha_{kj}, \alpha_{k\ell}$ are distinct, then by (2),  we have $Ad(G_k)\geq 2$.
\end{proof}

\begin{lemma}\label{lm:G'}
If $0<cap_{pos}(G)<+\infty$,
 $cap_{pos}(G)=cap_{pos}(G_{E\cap H})$.
\end{lemma}

\begin{proof}
(1) For any $k\in  \{1,\dots, s\}$ such that $k\notin E$, we have $\alpha_{k1}=\dots=\alpha_{km}$. The steady-state system defined in \eqref{eq:h1} and \eqref{eq:h} can be written as
\begin{align} h_1(x)~&=\;\left(\beta_{11}-\alpha_{11}\right)x_k^{\alpha_{k1}}\left(\sum\limits_{j=1}^m
\lambda_j\kappa_j\prod\limits_{i\neq k} x_i^{\alpha_{ij}}\right)\label{eq:sys1}\\
h_i(x) ~&= ~(\beta_{i1}-\alpha_{i1})x_1 - (\beta_{11}-\alpha_{11})x_i - c_{i-1}, \;\; 2\leq i\leq s. \label{eq:sys2}
\end{align}
Consider the network obtained by removing the species $X_k$ from network $G$, say $\tilde G$.
The steady-state system for $\tilde G$ is
\begin{align}
\tilde h_1(x)~&=\;\left(\beta_{11}-\alpha_{11}\right)\left(\sum\limits_{j=1}^m
\lambda_j\kappa_j\prod\limits_{i\neq k} x_i^{\alpha_{ij}}\right)\label{eq:sys3}\\
\tilde h_i(x) ~&= ~(\beta_{i1}-\alpha_{i1})x_1 - (\beta_{11}-\alpha_{11})x_i - c_{i-1}, \;\; i\in \{2, \ldots, s\}\backslash \{k\}
\label{eq:sys4}
\end{align}
(here, we assume $k\neq 1$, and if $k=1$, by  $0<cap_{pos}(G)<+\infty$ and by \cite[Lemma 5.11]{txzs2020}, one can similarly prove the conclusion by reordering the species).
Comparing the two steady-state systems \eqref{eq:sys1}--\eqref{eq:sys2} and  \eqref{eq:sys3}--\eqref{eq:sys4}, we have
$$cap_{pos}(G)=cap_{pos}(\tilde G).$$
(2)For any $k\in  \{1,\dots, s\}$ such that $k\notin H$, we have $\beta_{k1}-\alpha_{k1}=0$. Recall that by Assumption \ref{assumption1},
we have $\beta_{11}-\alpha_{11}\neq 0$. So, here we have $k\neq 1$.
 Hence, the steady-state system defined in \eqref{eq:h1} and \eqref{eq:h} becomes
\begin{align*} h_1(x)~&=\;\left(\beta_{11}-\alpha_{11}\right)\left(\sum\limits_{j=1}^m
\lambda_j\tilde \kappa_j\prod\limits_{i\neq k} x_i^{\alpha_{ij}}\right), \;\text{where}\; \tilde\kappa_j :=  \kappa_j\left(\frac{c_{k-1}}{\beta_{11}-\alpha_{11}}\right)^{\alpha_{kj}}.\\
h_i(x) ~&= ~(\beta_{i1}-\alpha_{i1})x_1 - (\beta_{11}-\alpha_{11})x_i - c_{i-1}, \;\; \;\; i\in \{2, \ldots, s\}\backslash \{k\}.\end{align*}
Consider the network obtained by removing the species $X_k$ from network $G$, say $\tilde G$. Then,
the above steady-state system is exactly the same with the steady-state system of $\tilde G$ if the rate constants of $\tilde G$ are denoted by
$\tilde \kappa_1, \ldots, \tilde \kappa_m$.
Therefore, $cap_{pos}(G)=cap_{pos}(\tilde G)$.

By (1) and (2), since $G_{E\cap H}$ is obtained by removing all the species indexed by $k\notin E\cap H$,
we have $cap_{pos}(G)=cap_{pos}(G_{E\cap H})$.
\end{proof}

\begin{lemma}\label{lm:eh}
If $0<cap_{pos}(G)<+\infty$, then  $E\cap H\neq \emptyset$.
\end{lemma}
\begin{proof}
By the proof of Lemma \ref{lm:G'}, we have
$cap_{pos}(G)=cap_{pos}(G_H)$. Assume $E\cap H=\emptyset$. Then for any
$k\in H$, we have $k\notin E$, i.e., $\alpha_{k1}=\dots=\alpha_{km}$. By \eqref{eq:h1}, the first equation in the steady-state system for the network
$G_H$ is
  $$h_1(x)=\left(\sum\limits_{j=1}^{m}
    \lambda_j\kappa_j\right)\left(\prod\limits_{k\in H}x_k^{\alpha_{k1}}\right)=0.$$
    Clearly, either $G_H$ has no positive steady states, or $G_H$ has infinitely many positive steady states,
  which is a contradiction to  $0<cap_{pos}(G_H)=cap_{pos}(G)<+\infty.$
\end{proof}

\begin{definition}\label{def:T-alter}\cite[Definition 3.4]{Joshi:Shiu:Multistationary}
  For a positive integer $T$, a \defword{$T$-alternating network} is a one-species network with exactly $T+1$ reactions and with arrow diagram $\rho\in\{\rightarrow,\leftarrow\}^{T+1}$ such that $\rho_i=\rightarrow$ if and only if $\rho_{i+1}=\leftarrow$ for all $i\in \{1,\dots,T\}$.
\end{definition}

\begin{lemma}\label{lm:neq2}
If $3\leq cap_{pos}(G)<+\infty$ and $|E\cap H|=1$, then $Ad(G)\geq 3$.
\end{lemma}

\begin{proof}
Let $E\cap H=\{k\}$, then by Lemma \ref{lm:G'}, we have $cap_{pos}(G)=cap_{pos}(G_k)$, and thus $3\leq cap_{pos}(G_k)<\infty$. By \cite[Theorem 3.6 2(b)]{Joshi:Shiu:Multistationary}, $G_k$ has a $3$-alternating subnetwork, and so, by Definition
\ref{def:T-alter}, Definition \ref{def:arrow} and Definition \ref{def:ad} (1)(2),  $Ad(G_k)\geq 3$. Finally, by Proposition \ref{prop:Ad} (1), we have $Ad(G)\geq Ad(G_k)\geq 3$.
\end{proof}

\begin{lemma}\label{lm:G=2}
  Suppose that $|\{j\in \{1, \ldots, m\}|\lambda_j>0\}|=1$ and $|E\cap H|=2$. Without loss of generality, let $E\cap H=\{1,2\}$. If both $\{\alpha_{11}, \dots, \alpha_{1m}\}$ and $\{\alpha_{21}, \dots, \alpha_{2m}\}$ contain only $2$ distinct values, then  $Ad(G)\geq 3$ or $cap_{pos}(G)\leq 2$.
\end{lemma}

\begin{proof}
Recall Assumption \ref{assumption2}. The hypothesis $|\{j\in \{1, \ldots, m\}|\lambda_j>0\}|=1$ means that $t=1$.
If $t=1$, then $\{1, \ldots, t\}=\{1\}$ and $\{t+1, \ldots, m\}=\{2, \ldots, m\}$.
  If there exist $i,j\in \{2,\dots,m\}$ $(i\neq j)$, such that $\alpha_{11}\neq \alpha_{1i}$ and $\alpha_{11}\neq \alpha_{1j}$, then by Proposition \ref{prop:Ad} (2) (note here $1\in E\cap H\subset H$), $Ad(G_1)\geq 2$. By Proposition \ref{prop:Ad} (3), $Ad(G_2)\geq 1$. So it follows from Proposition \ref{prop:Ad} (1) that $Ad(G)\geq Ad(G_1)+Ad(G_2)\geq 3$. Similarly, if there exists $i,j\in \{2,\dots,m\}$, $i\neq j$, such that $\alpha_{21}\neq \alpha_{2i}$ and $\alpha_{21}\neq \alpha_{2j}$, then we also have $Ad(G)\geq 3$.
 Below, we will consider the case that there exists only one index $j_1\in \{2,\dots,m\}$ such that $\alpha_{11}\neq \alpha_{1j_1}$ and
 there also exists only one  index $j_2\in \{2,\dots,m\}$ such that $\alpha_{21}\neq \alpha_{2j_2}$. Without loss of generality, we assume
 that $j_1=m$. So, we have   $\alpha_{11}=\dots=\alpha_{1 m-1}\neq\alpha_{1m}$. In the rest of the proof, we discuss the following two cases.

\textbf{Case 1:}  If we have $j_2=m$, then
$\alpha_{11}=\dots=\alpha_{1 m-1}\neq\alpha_{1m}$ and $\alpha_{21}=\dots=\alpha_{2 m-1}\neq\alpha_{2m}$.

    In this case, for the polynomial $h_1$ \eqref{eq:h1} defined in the steady-state system, we have
    \begin{align*}
      h_1(x)&~=~\left(\beta_{11}-\alpha_{11}\right)\left(\prod\limits_{k=3}^s x_k^{\alpha_{k1}}\right)
      \left(\sum\limits_{j=1}^{m-1}
    \lambda_j\kappa_jx_1^{\alpha_{11}}x_2^{\alpha_{21}}+\lambda_m\kappa_m
    x_1^{\alpha_{1m}}x_2^{\alpha_{2m}}\right)\\
     ~&~=~\left(\beta_{11}-\alpha_{11}\right)\left(\prod\limits_{k=3}^s x_k^{\alpha_{k1}}\right)
     x_1^{\alpha_{11}}x_2^{\alpha_{21}}\left(\sum\limits_{j=1}^{m-1}
    \lambda_j\kappa_j+\lambda_m\kappa_m
    x_1^{\alpha_{1m}-\alpha_{11}}x_2^{\alpha_{2m}-\alpha_{21}}\right).
    \end{align*}
    So the positive solutions to the  steady-state system \eqref{eq:h1}--\eqref{eq:h} are all the positive solutions to the following two equations:
    \begin{align}
    x_2&\;=\; {\mathcal C} x_1^{\frac{\alpha_{11}-\alpha_{1m}}{\alpha_{2m}-\alpha_{21}}}, \; \;\;\;\;\;\;\;\;\;\;\text{where}\; {\mathcal C}=\left(-\sum\limits_{j=1}^{m-1}\lambda_j\kappa_j/
    \lambda_m\kappa_m\right)^{\frac{1}{\alpha_{2m}-\alpha_{21}}} \label{eq:l1}\\
    x_2&\;=\;
    \frac{\beta_{21}-\alpha_{21}}{\beta_{11}-\alpha_{11}}x_1-
    \frac{c_1}{\beta_{11}-\alpha_{11}}. \label{eq:l2}
    \end{align}
  Note that a power function and a linear function have at most two intersection points of their graphs in the first quadrant. So, the above two equations have at most two common positive solutions.
  Therefore, we have $cap_{pos}(G)\leq 2$.

\textbf{Case 2:}
If we have $j_2\in \{2, \ldots, m-1\}$, then without loss of generality, we assume $j_2=2$. So, we have
 $\alpha_{11}=\dots=\alpha_{1 m-1}\neq\alpha_{1m}$ and $\alpha_{21}=\alpha_{23}=\dots=\alpha_{2 m}\neq\alpha_{22}$.

   In this case, for the polynomial $h_1$ \eqref{eq:h1} defined in the steady-state system, we have
    \begin{align*}
      h_1(x)&~=~\left(\beta_{11}-\alpha_{11}\right)\left(\prod\limits_{k=3}^s x_k^{\alpha_{k1}}\right)
      \left(\bar {\mathcal C} x_1^{\alpha_{11}}x_2^{\alpha_{21}}+\lambda_2\kappa_2
    x_1^{\alpha_{11}}x_2^{\alpha_{22}}+\lambda_m\kappa_m
    x_1^{\alpha_{1m}}x_2^{\alpha_{21}}\right)\\
    ~&~=~\left(\beta_{11}-\alpha_{11}\right)\left(\prod\limits_{k=3}^s x_k^{\alpha_{k1}}\right)
     x_1^{\alpha_{11}}x_2^{\alpha_{21}}\left(\bar {\mathcal C}+
     \lambda_2\kappa_2
    x_2^{\alpha_{22}-\alpha_{21}}+\lambda_m\kappa_m
    x_1^{\alpha_{1m}-\alpha_{11}}\right),
    \end{align*}
where $\bar {\mathcal C}=\lambda_1\kappa_1+\sum\limits_{j=3}^{m-1}
    \lambda_j\kappa_j$.
    So the positive solutions to the  steady-state system \eqref{eq:h1}--\eqref{eq:h} are all the positive solutions to the following two equations:
    \begin{align}
    x_2&\;=\; \left(-\frac{\bar {\mathcal C} +\lambda_m\kappa_m x_1^{\alpha_{1m}-\alpha_{11}}}{\lambda_2\kappa_2}\right)^{\frac{1}{\alpha_{22}-\alpha_{21}}}, \label{eq:l3} \\
    x_2&\;=\;\frac{\beta_{21}-\alpha_{21}}{\beta_{11}-\alpha_{11}}x_1-
    \frac{c_1}{\beta_{11}-\alpha_{11}}.
    \label{eq:l4}
\end{align}
    Let $$F(x_1)=\left(-\frac{\bar {\mathcal C} +\lambda_m\kappa_m x_1^{\alpha_{1m}-\alpha_{11}}}{\lambda_2\kappa_2}\right)^{\frac{1}{\alpha_{22}-\alpha_{21}}}
    -\frac{\beta_{21}-\alpha_{21}}{\beta_{11}-\alpha_{11}}x_1+
    \frac{c_1}{\beta_{11}-\alpha_{11}}.$$
    It is straightforward to check that $F''(x_1)$ does not change sign when
    $x_1>0$ and $-\frac{\bar {\mathcal C} +\lambda_m\kappa_m x_1^{\alpha_{1m}-\alpha_{11}}}{\lambda_2\kappa_2}>0$. However, if the equations \eqref{eq:l3}-\eqref{eq:l4}  have at least three common positive solutions such that
    $x_1>0$ and $x_2>0$, then
       $F''(x_1)=0$ has at least one solution.
       So,  the equations \eqref{eq:l3}-\eqref{eq:l4} have at most two common positive solutions.    Above all, we have $cap_{pos}(G)\leq 2$.
\end{proof}

{\bf Proof of Theorem \ref{thm:ad3}.}
\begin{proof}
  Since $cap_{pos}(G)\geq 3>0$, by Lemma
  \ref{lm:eh},
  we have $E\cap H\neq \emptyset$.
    By Proposition \ref{prop:Ad} (4), if $|E\cap H|\geq 3$, then $Ad(G)\geq 3$.    By Lemma \ref{lm:neq2}, if $|E\cap H|=1$, then $Ad(G)\geq 3$.
    Below, we suppose that $|E\cap H|=2$. Without loss of generality,  we assume that $E\cap H=\{1,2\}$.
    By Proposition \ref{prop:Ad} (3), for $k=1$, we have $Ad(G_k)\geq 1$. By Definition \ref{def:ad}, there exist
    $i\in \{1, \ldots, t\}$ and $j\in \{t+1, \ldots, m\}$ such that   $\alpha_{1i}\neq \alpha_{1j}$. Without loss of generality, we assume that
    $i=1$ and $j=m$, i.e.,     $\alpha_{11}\neq \alpha_{1m}$.

    First, we prove that if $2\leq t\leq m-2$, then $Ad(G)\geq 3$. In fact, if $2\leq t\leq m-2$,
    then $2\in \{1, \ldots, t\}$ and $m-1\in \{t+1, \ldots, m\}$.
    If  $\alpha_{11}\neq \alpha_{1,m-1}$ or $\alpha_{12}\neq \alpha_{1m}$, then
    by Proposition \ref{prop:Ad} (2), $Ad(G_1)\geq 2$. If $\alpha_{12}= \alpha_{1m}$ and $\alpha_{1,m-1}= \alpha_{11}$, then $\alpha_{12}\neq \alpha_{1,m-1}$, and thus, by  Proposition \ref{prop:Ad} (2), $Ad(G_1)\geq 2$.
    Note that by Proposition \ref{prop:Ad} (3), we have $Ad(G_2)\geq 1$ since $2\in E\cap H$.
    So by Proposition \ref{prop:Ad} (1), $Ad(G)\geq Ad(G_1)+Ad(G_2)\geq 2+1=3$.

    Second, we  prove that if $t=1$ or $t=m-1$, we have $Ad(G)\geq 3$. We only prove the conclusion when $t=1$, and one can similarly prove the conclusion if $t=m-1$. By Proposition \ref{prop:Ad} (5), if $\{\alpha_{11},\dots,\alpha_{1m}\}$ or $\{\alpha_{21},\dots, \alpha_{2m}\}$ contains at least 3 distinct values, then $Ad(G_1)\geq2$ or $Ad(G_2)\geq2$.
    Recall again that by Proposition \ref{prop:Ad} (3), we have  $Ad(G_1)\geq1$ and $Ad(G_2)\geq1$ since $E\cap H=\{1,2\}$.
    So  by Proposition \ref{prop:Ad} (1), $Ad(G)\geq Ad(G_1)+Ad(G_2)\geq 3$. If both $\{\alpha_{11},\dots,\alpha_{1m}\}$ and $\{\alpha_{21},\dots, \alpha_{2m}\}$ contain only 2 distinct values, then by Lemma \ref{lm:G=2} and by the hypothesis that $cap_{pos}(G)\geq 3$, we have $Ad(G)\geq 3$.
\end{proof}

\section{Bi-reaction networks admitting three positive steady states}\label{sec:bi}
In this section, we focus on the bi-reaction network $G$:
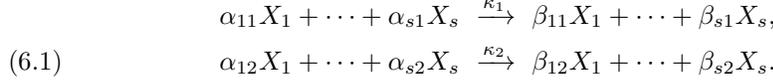
\begin{align}\label{eq:2network}
\alpha_{11}X_1 +
 \dots +
\alpha_{s1}X_s
&~ \xrightarrow{\kappa_1} ~
\beta_{11}X_1 +
 \dots +
\beta_{s1}X_s,
\notag \\
 \alpha_{12}X_1 +
 \dots +
\alpha_{s2}X_s
&~ \xrightarrow{\kappa_2} ~
\beta_{12}X_1 +
 \dots +
\beta_{s2}X_s.
\end{align}
For $k\in \{1, \ldots, s\}$, we define the following notions
\begin{align}\label{eq:ar}
\alpha_k:=\alpha_{k1}-\alpha_{k2},\; \;\;
\gamma_k:=\beta_{k1}-\alpha_{k1}.  
\end{align}
Without loss of generality, we make the following assumption, which gives a partition of the set $\{1, \ldots, s\}$:
\begin{align}\label{eq:s}
  \text{for}\; j \in S_1 := \{ 1, \cdots, i_1 -1\},\;
  \alpha_j > 0,\; \gamma_j > 0,\notag \\
  \text{for} \; j \in S_2 := \{ i_1, \cdots, i_2 -1\},\;
  \alpha_j < 0, \; \gamma_j < 0,\notag \\
  \text{for} \; j \in S_3 := \{ i_2, \cdots, i_3- 1\},\;
  \alpha_j >0,\; \gamma_j < 0,\notag \\
  \text{for} \; j \in S_4 := \{ i_3, \cdots, i_4 -1\},\;
  \alpha_j < 0,\; \gamma_j > 0, \text{and} \notag \\
  \text{for} \; j \in S_5 := \{ i_4, \cdots, s\},\;
  \alpha_j = 0\; \text{or}\; \gamma_j = 0.
\end{align}
For $k\in \{1, 2, 3, 4\}$, we define $sign_{\alpha}(k) := sign(\alpha_i)$ and $sign_{\alpha\gamma}(k) := sign(\alpha_i\gamma_i)$, where $i$ can be any index in $S_k$. By the definition of $S_k$, we have
$$sign_{\alpha}(1)>0, sign_{\alpha}(2)<0, sign_{\alpha}(3)>0, sign_{\alpha}(4)<0,$$
$$sign_{\alpha\gamma}(1)>0, sign_{\alpha\gamma}(2)>0, sign_{\alpha\gamma}(3)<0, sign_{\alpha\gamma}(4)<0.$$
Clearly, for any three indices $j, k, \ell\in \{1, 2, 3, 4\}$, there exists a unique pair of indices, say $(k, \ell)$,  such that
$sign_{\alpha}(k)=-sign_{\alpha}({\ell})=-sign_{\alpha}(j)$, and $sign_{\alpha\gamma}(\ell)=-sign_{\alpha\gamma}(k)=-sign_{\alpha\gamma}(j)$.
\begin{theorem}\label{thm:main}
Given a bi-reaction network $G$ \eqref{eq:2network} with a one-dimensional stoichiometric subspace,  suppose $cap_{pos}(G)<+\infty$.
\begin{itemize}
\item[(a)] If only one of the four sets $S_1, S_2, S_3, S_4$ is non-empty, then $cap_{pos}(G)<3$.
\item[(b)] If there are exactly two of the four sets $S_1,S_2,S_3,S_4$ are non-empty, then $cap_{pos}(G)\geq 3$ if and only if

  (1) both the sequence $\{\alpha_k\}_{k=1}^{i_4-1}$ and $\{\alpha_k\gamma_k\}_{k=1}^{i_4-1}$ change signs,

   (2) $\sum\limits_{i\in S_k}|\alpha_{i}|>\min\limits_{i\in S_{\ell}}\{|\alpha_i|\}$ and $\sum\limits_{i\in S_{\ell}}|\alpha_{i}|>\min\limits_{i\in S_k}\{|\alpha_i|\}$, where $S_k$ and $S_{\ell}$ are the two non-empty sets.
\item[(c)] If there are exactly three of the four sets $S_1, S_2, S_3, S_4$ are non-empty, say $S_j$, $S_k$, and $S_{\ell}$, then $cap_{pos}(G)\geq 3$ if and only if $\sum\limits_{i\in S_{\ell}}|\alpha_{i}|>\min\limits_{i\in S_k}\{|\alpha_i|\}$, where $k$ and $\ell$ are assumed to be the two indices such that $sign_{\alpha}(k)=-sign_{\alpha}({\ell})=-sign_{\alpha}(j)$, and $sign_{\alpha\gamma}(\ell)=-sign_{\alpha\gamma}(k)=-sign_{\alpha\gamma}(j)$.
\item[(d)] If all the four sets $S_1, S_2, S_3, S_4$ are non-empty, then $cap_{pos}(G)\geq 3$ if and only if at least one of the following four conditions holds
 \begin{align}
\sum\limits_{i\in S_4}|\alpha_{i}|>\min\limits_{i\in S_1}\{|\alpha_i|\},\;\;
\sum\limits_{i\in S_1}|\alpha_{i}|> \min\limits_{i\in S_4}\{|\alpha_i|\}, \label{eq:cond1}\\
\sum\limits_{i\in S_3}|\alpha_{i}|> \min\limits_{i\in S_2}\{|\alpha_i|\},\;\;
 \sum\limits_{i\in S_2}|\alpha_{i}|> \min\limits_{i\in S_3}\{|\alpha_i|\}.\label{eq:cond2}
 \end{align}
\end{itemize}
\end{theorem}

\begin{example}These examples illustrate how Theorem \ref{thm:main} works.
  \begin{itemize}
    \item [(a)] Consider the following network: $$G_a:\begin{aligned}
      2X_1+2X_2&\rightarrow 3X_1+3X_2\\
      X_1+X_2&\rightarrow  0
    \end{aligned}$$
It is straightforward to check that
$$\begin{aligned}
      &\alpha_1=1,~ \alpha_2=1,\;
      \gamma_1=1,~ \gamma_2=1. \\
      &S_1=\{1,2\},~ S_2=\emptyset,~ S_3=\emptyset, ~ S_4=\emptyset,~ S_5=\emptyset.
    \end{aligned}$$
So by Theorem \ref{thm:main} (a), we have $cap_{pos}(G_a)< 3$.
    \item [(b)]
    Consider the following network:$$G_b:\begin{aligned}
      3X_1+2X_2+X_3&\rightarrow 4X_1+3X_2+2X_3\\
      X_1+X_2+3X_3&\rightarrow  2X_3
    \end{aligned}$$
It is straightforward to check the following results:
$$\begin{aligned}
      &\alpha_1=2,~ \alpha_2=1,~ \alpha_3=-2 ~(\text{change signs}).\\
      &\gamma_1=1,~ \gamma_2=1,~ \gamma_3=1.\\
      &\alpha_1\gamma_1=2,~ \alpha_2\gamma_2=1,~ \alpha_3\gamma_3=-2 ~(\text{change signs}).\\
      &S_1=\{1,2\},~ S_2=\emptyset,~ S_3=\emptyset, ~ S_4=\{3\},~ S_5=\emptyset.\\
      &\{k,\ell\}=\{1,4\}.\\
      &\sum_{i\in S_1}|\alpha_i|=|\alpha_1|+|\alpha_2|=3>2=|\alpha_3|=\min_{i\in S_4}\{|\alpha_i|\}.\\
      &\sum_{i\in S_4}|\alpha_i|=|\alpha_3|=2>1=|\alpha_2|=\min_{i\in S_1}\{|\alpha_i|\}.
    \end{aligned}$$
So by Theorem \ref{thm:main} (b), we have $cap_{pos}(G_b)\geq 3$. One can also check that for                           $c_1=19999/2,~c_2=9999,~\kappa_1=26879/4294967296,$ and  $\kappa_2=1/2$,
    the network has $3$ positive steady states: $$\begin{aligned}
      &x^{(1)}=(9999.50,~0.00020,~0.50),\\
      &x^{(2)}=(11716.81,~1717.31,~1717.81),\\
      &x^{(3)}=(68177.66,~58178.16,~58178.66).
    \end{aligned}$$
\item [(c)] Consider the following network:$$G_c:\begin{aligned}
      2X_1+X_2&\rightarrow 3X_1+X_3\\
      X_1+2X_2+2X_3&\rightarrow 3X_2+X_3\
    \end{aligned}$$
It is straightforward to check the following results:
    $$\begin{aligned}
      &\alpha_1=1,~ \alpha_2=-1, ~\alpha_3=-2.\\
      &\gamma_1=1,~ \gamma_2=-1, ~\gamma_3=1.\\
      &\alpha_1\gamma_1=1,~ \alpha_2\gamma_2=1,~ \alpha_3\gamma_3=-2.\\
      &S_1=\{1\},~ S_2=\{2\},~ S_3=\emptyset, ~ S_4=\{3\},~ S_5=\emptyset.\\
      &sign_{\alpha}(1)=-sign_{\alpha}(2)=-sign_{\alpha}(4),\\
      &sign_{\alpha\gamma}(4)=-sign_{\alpha\gamma}(1)=-sign_{\alpha\gamma}(2).\\
      &k=1,~ \ell=4, ~\min_{i\in S_1}\{|\alpha_i|\}=|\alpha_1|=1<2=|\alpha_3|=\sum_{i\in S_4}|\alpha_i|.
    \end{aligned}$$ So by Theorem \ref{thm:main} (c), we have $cap_{pos}(G_c)\geq 3$. One can also check that for                           $c_1=-10001,~c_2=-1,~\kappa_1=102624395269/16384,$ and  $\kappa_2=1/2$,
    the network has $3$ positive steady states: $$\begin{aligned}
      &x^{(1)}=(0.00080,~10001.00,~1.00),\\
      &x^{(2)}=(1465.72,~8535.28,~1466.72),\\
      &x^{(3)}=(8533.28,~1467.72,~8534.28).
    \end{aligned}$$
    \item [(d)] Consider the following network:$$G_d:\begin{aligned}
      3X_1+X_2+X_3&\rightarrow 4X_1+X_4\\
      X_1+2X_2+X_4&\rightarrow 3X_2+X_3
    \end{aligned}$$
It is straightforward to check that:
    $$\begin{aligned}
      &\alpha_1=2,~ \alpha_2=-1, ~\alpha_3=1,~\alpha_4=-1.\\
      &\gamma_1=1,~ \gamma_2=-1, ~\gamma_3=-1,~\gamma_4=1.\\
      &S_1=\{1\},~ S_2=\{2\},~ S_3=\{3\}, ~ S_4=\{4\},~ S_5=\emptyset.\\
      &\sum_{i\in S_1}|\alpha_i|=|\alpha_1|=2>1=|\alpha_4|=\min_{i\in S_4}\{|\alpha_i|\}.
    \end{aligned}$$ So by Theorem \ref{thm:main} (d), we have $cap_{pos}(G_d)\geq 3$.
    One can also check that for                           $c_1=-10003,~c_2=-10002,~c_3=1,~\kappa_1=262251/4194304,$ and  $\kappa_2=1/2$,
    the network has $3$ positive steady states: $$\begin{aligned}
      &x^{(1)}=(1.17,~10001.83,~10000.83,~0.17),\\
      &x^{(2)}=(6.83,~9996.17,~9995.17,~5.83),\\
      &x^{(3)}=(10002.00,~1.00,~0.00080,~10001.00).
    \end{aligned}$$
  \end{itemize}
  In the above networks, we find the witnesses for admitting $3$ positive steady states by the proofs of Theorem \ref{thm:main}. See the subsection.
\end{example}

\subsection{Proof of Theorem \ref{thm:main}}
Throughout this section, we assume that any network $G$ mentioned in the lemmas has the form \eqref{eq:2network}, and it has a one-dimensional stoichiometric subspace.

We introduce a new variable $z$, and a new parameter $d_1$ such that
\begin{align}\label{eq:x1}
x_1 = \gamma_1 z + d_1.
\end{align}
Then, the conservation laws (see $h_i$ defined in \eqref{eq:h}) can be written as
\begin{align}\label{eq:xi}
x_i = \gamma_i z + d_i, \;\;\;\text{where}\;\;\; d_i := \frac{\gamma_i}{\gamma_1}d_1 - \frac{c_{i-1}}{\gamma_1}, \; i\in \{2, \ldots, s\}.
\end{align}
Also, we recall \eqref{eq:scalar} that $\lambda_1=1$, and $\lambda_2$ is a non-zero real number.

\begin{lemma}\label{lm:g}
  A network $G$ admits at least $N$ $(N\geq 0)$ positive steady states if and only if
  there exist  rate constants $\kappa_1, \kappa_2\in {\mathbb R}_{>0}$ and  real numbers $d_1,\dots, d_s$ such that the
  equation
  \begin{align}\label{eq:lmg}
  \sum\limits_{k=1}^{s}\alpha_k\ln(\gamma_kz+d_k) \; =\; K, \;
  \text{where}\;K:=\ln(-\lambda_2\kappa_2/\lambda_1\kappa_1)
  \end{align}
  has at least $N$ different solutions for $z$ in the open interval
  {\footnotesize
  \begin{align}
 &I:=(a, M), \text{where}  \label{eq:interval}\\
 &a:=
 \begin{cases}
  \max\limits_{\{k|\gamma_k>0\}}\{-\frac{d_{k}}{\gamma_k}\},& \text{if} ~{\{k|\gamma_k>0\}}\neq \emptyset\\
  -\infty , &\text{if}~ {\{k|\gamma_k>0\}}= \emptyset
\end{cases}
 ,
  M:=
\begin{cases}
  \min\limits_{\{k|\gamma_k<0\}}\{-\frac{d_{k}}{\gamma_k}\},& \text{if} ~{\{k|\gamma_k<0\}}\neq \emptyset\\
  +\infty , &\text{if}~ {\{k|\gamma_k<0\}}= \emptyset
\end{cases}.  \label{eq:am}
\end{align}
}
  \end{lemma}

\begin{proof}
By the steady-state system \eqref{eq:h1}--\eqref{eq:h} and by \eqref{eq:x1}--\eqref{eq:xi}, the network $G$ admits at least 3 positive steady states if and only if there exists a rate-constant $\kappa=(\kappa_1, \kappa_2)\in {\mathbb R}^2_{>0}$ and real numbers $d_1,\dots,d_{s}$ such that the equation
\begin{align}\label{eq:2f}
(\beta_{11}-\alpha_{11})\sum_{j=1}^2\lambda_j\kappa_j\prod\limits_{k=1}^s\left(\gamma_kz+d_k\right)^{\alpha_{kj}}=0.
\end{align}
 has at least 3 different solutions for $z$ such that for any $k=1, \ldots, s$, $\gamma_kz+d_k>0$.
 Notice that  $\gamma_kz+d_k>0$ for any $k$ if and only if
$z$ belongs to the open interval $I$ defined in \eqref{eq:interval}.
Since $\beta_{11}-\alpha_{11}\neq 0$, the equation \eqref{eq:2f} is equivalent to
$$\lambda_1\kappa_1\prod\limits_{k=1}^s\left(\gamma_kz+d_k\right)^{\alpha_{k1}}+
\lambda_2\kappa_2\prod\limits_{k=1}^s\left(\gamma_kz+d_k\right)^{\alpha_{k2}}=0.$$
For any $z\in I$, the above equation is equivalent to
$$\prod\limits_{k=1}^s\left(\gamma_kz+d_k\right)^{\alpha_{k}}
=-\frac{\lambda_2\kappa_2}{\lambda_1\kappa_1}.$$
Here, recall that $\alpha_k=\alpha_{k1}-\alpha_{k2}$, see \eqref{eq:ar}.
Take logarithm on both sides, we have $$\sum\limits_{k=1}^s\alpha_k\ln\left(\gamma_kz+d_k\right)=
\ln\left(-\frac{\lambda_2\kappa_2}{\lambda_1\kappa_1}\right).$$
So, we have the conclusion.
\end{proof}


\begin{lemma}\label{lm:sign} We have the
following statements.
\begin{enumerate}
  \item[(a)] If the network $G$ admits at least $2$ positive steady states, then the numbers in the sequence $\{\alpha_k\gamma_k\}_{k=1}^{i_4-1}$ change signs.
  \item[(b)] If the network $G$ admits at least $3$ positive steady states, then the numbers in the sequence $\{\alpha_k\}_{k=1}^{i_4-1}$change signs.
\end{enumerate}
\end{lemma}

\begin{proof}
Define
\begin{align}\label{eq:g}
g(z)\;:= \;  \sum\limits_{k=1}^{s}\alpha_k\ln(\gamma_kz+d_k).
\end{align}

(a) By Lemma \ref{lm:g}, if the network $G$ admits at least $2$ positive steady states, then
there exists $K\in {\mathbb R}$ such that $g(z)=K$ has at least 2 different solutions in the interval $I$ defined in \eqref{eq:interval}. It follows from Rolle's Theorem that $g'(z)=0$ has at least $1$ solution in $I$. Note that $$g'(z)=\sum\limits_{k=1}^{s}\frac{\alpha_k\gamma_k}{\gamma_kz+d_k}=\sum\limits_{k=1}^{i_4-1}\frac{\alpha_k\gamma_k}{\gamma_kz+d_k},$$
where the last equality follows from \eqref{eq:s} that for any $k\in \{i_4, \ldots, s\}$, $\alpha_k\gamma_k=0$.
Note also, if $z\in I$, then for any $k\in \{1, \ldots, s\}$, $\gamma_kz+d_k>0$.
  So, if  $g'(z)=0$ has a solution in $I$, then the numbers in the sequence $\{\alpha_k\gamma_k\}_{k=1}^{i_4-1}$ change signs.

(b) By Lemma \ref{lm:g}, if the network $G$ admits at least $3$ positive steady states, then
there exists $K\in {\mathbb R}$ such that $g(z)=K$ has at least 3 different solutions in the interval $I$ defined in \eqref{eq:interval}.
So, $g''(z)=0$ has at least $1$ solution in $I$.
 Note that $$g''(z)=-\sum\limits_{k=1}^{i_4-1}\frac{\alpha_k\gamma_k^2}{\left(\gamma_kz+d_k\right)^2}.$$
  So, if $g''(z)$ has a solution in $I$, then the numbers in the sequence $\{\alpha_k\}_{k=1}^{i_4-1}$ change signs.
\end{proof}

\begin{lemma}\label{lm:3roots}
  Given a univariate $C^2$-function $g(z)$ defined on an open interval $(a,M)\subset {\mathbb R}$, suppose $\lim\limits_{z\rightarrow M}g(z)=\infty$. If there exists  $z^*\in (a,M)$ such that $g'(z^*)=0$ and $$g''(z^*)\lim\limits_{z\rightarrow M}g(z)<0,$$ then there exists $K \in \mathbb{R}$ such that $g(z)=K$ has at least 3 different solutions in $(a,M)$.
\end{lemma}
\begin{proof}
Since $g''(z^*)\lim\limits_{x\rightarrow M}g(z)<0$, we have $g''(z^*)\neq 0$. Without loss of generality, we assume that $g''(z^*)<0$. And thus,
$\lim\limits_{z\rightarrow M}g(z)=+\infty$. Since $g''(z)$ is continuous, there exists $\delta>0$ such that $g''(z)<0$ in $(z^*-\delta,z^*+\delta)\subset (a,M)$. So, $g'(z)$  strictly decreases in $(z^*-\delta,z^*+\delta)$. Note that $g'(z^*)=0$. So, $g'(z)>0$ in $(z^*-\delta,z^*)$, and $g'(z)<0$ in $(z^*,z^*+\delta)$. Hence, $g(z)$ strictly  increases in $(z^*-\delta,z^*)$ and strictly  decreases in $(z^*,z^*+\delta)$. Let $\epsilon=\min\{g(z^*)-g(z^*-\delta),g(z^*)-g(z^*+\delta)\}$ and let $K= g(z^*)-\epsilon/2$, then $g(z)=K$ has $2$ solutions in $(z^*-\delta,z^*+\delta)$. Since $\lim\limits_{z\rightarrow M}g(z)=+\infty$ and $g(z^*+\delta)<K$, $g(z)=K$ has at least $1$ solution in $(z^*+\delta,M)$. Overall, $g(z)=K$ has at least $3$ different solutions in $(a,M)$.
\end{proof}

For simplicity, for any $k\in \{1,\dots, s\}$, we define
\begin{align}\label{eq:ab}
b_k(z):=\left|\frac{\gamma_k}{\gamma_k z+d_k}\right|.
\end{align}
 By \eqref{eq:s}, for $g(z)$ defined in \eqref{eq:g},   for any $z$ in the interval $I$ defined in \eqref{eq:interval}, we have
\begin{align}
 g'(z)&=\sum\limits_{k\in S_1}|\alpha_k|b_k+\sum\limits_{k\in S_2}|\alpha_k|b_k-\sum\limits_{k\in S_3}|\alpha_k|b_k-\sum\limits_{k\in S_4}|\alpha_k|b_k, \label{eq:gd1}\\
 g''(z)&=-\sum\limits_{k\in S_1}|\alpha_k|b_k^2+\sum\limits_{k\in S_2}|\alpha_k|b_k^2-\sum\limits_{k\in S_3}|\alpha_k|b_k^2+\sum\limits_{k\in S_4}|\alpha_k|b_k^2.\label{eq:gd2}
 \end{align}

\begin{lemma}\label{lm:g''}
 Suppose that $S_2=S_3=\emptyset$, $S_1\neq\emptyset,S_4\neq\emptyset$.
 Let $g(z)$ be the function defined in \eqref{eq:g}, and let $I$ be the interval defined in \eqref{eq:interval}.
 If there exists $z^*\in I$ such that $g'(z^*)=0$, then for $z=z^*$, we have the following inequalities:
 \begin{align}\label{eq:ineq}
   -\frac{1}{\min\limits_{k\in S_1}\{|\alpha_k|\}}+\frac{1}{\sum\limits_{k\in S_4}|\alpha_k|}\leq
 \frac{g''(z)}{\left(\sum\limits_{k\in S_1}|\alpha_k|b_k\right)^2}\leq -\frac{1}{\sum\limits_{k\in S_1}|\alpha_k|}+\frac{1}{\min\limits_{k\in S_4}\{|\alpha_k|\}}.
 \end{align}
 More than that, if the first equality holds, then $S_1$ is singleton, and if the second equality holds, then $S_4$ is singleton.
\end{lemma}
\begin{proof}
Notice that by \eqref{eq:interval} and \eqref{eq:ab}, for any $z\in I$, $b_k(z)$ is well-defined and $b_k(z)>0$.
So,  by Cauchy's inequality, we have
\begin{align}\label{eq:cauchy}
\left(\sum\limits_{k\in S_i}|\alpha_k|b_k^2\right)\left(\sum\limits_{k\in S_i}|\alpha_k|\right)\geq\left(\sum\limits_{k\in S_i}|\alpha_k|b_k\right)^2, \;\;i=1, 4.
\end{align}
On the other hand, \begin{align}\label{eq:cauchynext}
 \min\limits_{k\in S_i}\{|\alpha_k|\} \sum\limits_{k\in S_i}|\alpha_k|b_k^2\leq \sum\limits_{k\in S_i}|\alpha_k|^2b_k^2  \leq \left(\sum\limits_{k\in S_i}|\alpha_k|b_k\right)^2,  \;\;i=1, 4.
\end{align}
Note that the last equality holds if and only if $S_i$ is singleton.
If $S_2=S_3=\emptyset$, then
by  \eqref{eq:gd2}, \eqref{eq:cauchy} and \eqref{eq:cauchynext},
we have
$$g''(z)=-\sum\limits_{k\in S_1}|\alpha_k|b_k^2+\sum\limits_{k\in S_4}|\alpha_k|b_k^2$$
and
{\tiny
$$-\frac{\left(\sum\limits_{k\in S_1}|\alpha_k|b_k\right)^2}{\min\limits_{k\in S_1}\{|\alpha_k|\}}+\frac{\left(\sum\limits_{k\in S_4}|\alpha_k|b_k\right)^2}{\sum\limits_{k\in S_4}|\alpha_k|}\leq -\sum\limits_{k\in S_1}|\alpha_k|b_k^2+\sum\limits_{k\in S_4}|\alpha_k|b_k^2\leq -\frac{\left(\sum\limits_{k\in S_1}|\alpha_k|b_k\right)^2}{\sum\limits_{k\in S_1}|\alpha_k|}+\frac{\left(\sum\limits_{k\in S_4}|\alpha_k|b_k\right)^2}{\min\limits_{k\in S_4}\{|\alpha_k|\}}.$$
}Since $g'(z^*)=0$, for $z=z^*$, we have $\sum\limits_{k\in S_1}|\alpha_k|b_k=\sum\limits_{k\in S_4}|\alpha_k|b_k$. So, for $z=z^*$, we have the inequality
\eqref{eq:ineq}.
\end{proof}

\begin{lemma}\label{lm:14}
Suppose $cap_{pos}(G)<+\infty$.
 If $S_2=S_3=\emptyset$, $S_1\neq\emptyset$, and $S_4\neq\emptyset$, then
    $\sum\limits_{i\in S_1}|\alpha_{i}|\neq \sum\limits_{i\in S_4}|\alpha_i|$.
\end{lemma}
\begin{proof}
Suppose $\sum\limits_{k\in S_1}|\alpha_k|=\sum\limits_{k\in S_4}|\alpha_k|$.
 It is straightforward to check that if we choose $d_1, \ldots, d_{i_4}$ such that
$\frac{d_1}{\gamma_1}=\dots=\frac{d_{i_4}}{\gamma_{i_4}}$, then $g(z)=\sum\limits_{k=1}^{i_4}|\alpha_k|\ln(d_k)$.
Hence, we can find a real number $K$ such that $g(z)=K$ for any $z$ in the interval $I$ defined in \eqref{eq:interval} (note here, $I=(-\frac{d_1}{\gamma_1}, +\infty)$).
So, by Lemma \ref{lm:g}, we have $cap_{pos}(G)=+\infty$,  which is a contradiction to the hypothesis that $cap_{pos}(G)<+\infty$.
\end{proof}

\begin{lemma}\label{lm:necessary}
Suppose $cap_{pos}(G)<+\infty$.
   Suppose that $S_2=S_3=\emptyset$, $S_1\neq\emptyset$, and $S_4\neq\emptyset$. If the network $G$ admits at least 3 positive steady states, then
   $\sum\limits_{i\in S_4}|\alpha_{i}|>\min\limits_{i\in S_1}\{|\alpha_i|\}$ and $\sum\limits_{i\in S_1}|\alpha_{i}|>\min\limits_{i\in S_4}\{|\alpha_i|\}$.
\end{lemma}
\begin{proof}
Let $g(z)$ be the function defined in \eqref{eq:g}.
By Lemma \ref{lm:g},
there exists $K\in {\mathbb R}$  such that $g(x)=K$ has at least 3 different solutions in the interval $I$ defined in \eqref{eq:interval}. By Rolle's Theorem, $g'(z)=0$ has at least 2 different solutions in $I$, i.e., there exist $z_1,z_2\in I$ $(z_1\neq z_2)$ such that $g'(z_1)=g'(z_2)=0$ and $g''(z_1)g''(z_2)\leq 0$.
Without loss of generality,
we assume that $g''(z_1)\geq 0$ and $g''(z_2)\leq 0$.
Recall $b_k(z)$ defined in \eqref{eq:ab}.
Hence, by Lemma \ref{lm:g''}, we have
\begin{align}\label{eq:nec1}
-\frac{1}{\min\limits_{k\in S_1}\{|\alpha_k|\}}+\frac{1}{\sum\limits_{k\in S_4}|\alpha_k|}\leq \frac{g''(z_2)}{\left(\sum\limits_{k\in S_1}|\alpha_k|b_k(z_2)\right)^2} \leq 0, \;\text{and}\;
\end{align}
\begin{align}\label{eq:nec2}
 0\leq \frac{g''(z_1)}{\left(\sum\limits_{k\in S_1}|\alpha_k|b_k(z_1)\right)^2}\leq -\frac{1}{\sum\limits_{k\in S_1}|\alpha_k|}+\frac{1}{\min\limits_{k\in S_4}\{|\alpha_k|\}}.
\end{align}
So, we have  $\sum\limits_{k\in S_4}|\alpha_{k}|\geq \min\limits_{k\in S_1}\{|\alpha_k|\}$ and $\sum\limits_{k\in S_1}|\alpha_{k}|\geq \min\limits_{k\in S_4}\{|\alpha_k|\}$.

If $\sum\limits_{k\in S_4}|\alpha_{k}|=\min\limits_{k\in S_1}\{|\alpha_k|\}$, then the equalities in \eqref{eq:nec1} hold. By Lemma \ref{lm:g''}, we have
$S_1$ is singleton, and hence, we have $\sum\limits_{k\in S_4}|\alpha_{k}|=\min\limits_{k\in S_1}\{|\alpha_k|\}=\sum\limits_{k\in S_1}|\alpha_{k}|$. This is a contradiction to Lemma \ref{lm:14}. So, we have $\sum\limits_{k\in S_4}|\alpha_{k}|>\min\limits_{k\in S_1}\{|\alpha_k|\}$.  Similarly, we can show that
$\sum\limits_{k\in S_1}|\alpha_{k}|> \min\limits_{k\in S_4}\{|\alpha_k|\}$.
\end{proof}

\begin{lemma}\label{lm:sufficient}
  Suppose that $S_2=S_3=\emptyset$, $S_1\neq\emptyset$, and $S_4\neq\emptyset$.
   \begin{enumerate}
     \item[(1)]If $\sum\limits_{k\in S_1}|\alpha_k|<\sum\limits_{k\in S_4}|\alpha_k|$ and $\sum\limits_{k\in S_1}|\alpha_{k}|>\min\limits_{k\in S_4}\{|\alpha_k|\}$, then the network $G$ admits at least 3 different positive steady states.
     \item[(2)]If $\sum\limits_{k\in S_4}|\alpha_k|<\sum\limits_{k\in S_1}|\alpha_k|$ and $\sum\limits_{k\in S_4}|\alpha_{k}|>\min\limits_{k\in S_1}\{|\alpha_k|\}$, then the network $G$ admits at least 3 different positive steady states.
   \end{enumerate}
\end{lemma}
\begin{proof}
We only need to prove (1), and the statement (2) holds by symmetry. Let $g(z)$ be the function defined in \eqref{eq:g}, and let $I=(a, M)$ be the interval defined in \eqref{eq:interval}.
 By Lemma \ref{lm:g}, we only need to prove that we can choose
 $d_1, \ldots, d_s$ such that for some $K\in {\mathbb R}$,
 $g(z)=K$ has at least three different solutions in $I$.
 If  $S_2=S_3=\emptyset$, $S_1\neq\emptyset$, and $S_4\neq\emptyset$, then by \eqref{eq:s} and \eqref{eq:am}, we have $a=\max\limits_{k\in S_1\cup S_4}\{-\frac{d_{k}}{\gamma_k}\}$ and $M=+\infty$.
 By \eqref{eq:g}, we have
\begin{align*}
  g(z)&=\sum\limits_{k\in S_1}|\alpha_k|\ln(\gamma_k z+d_k)-\sum\limits_{k\in S_4}|\alpha_k|\ln(\gamma_k z+d_k)\\
&\sim \left(\sum\limits_{k\in S_1}|\alpha_k|-\sum\limits_{k\in S_4}|\alpha_k|\right)\ln(z) \;\;\; (z\rightarrow +\infty).
\end{align*}
So, if $\sum\limits_{k\in S_1}|\alpha_k|<\sum\limits_{k\in S_4}|\alpha_k|$, then we have $\lim\limits_{z\rightarrow M}g(z)=\lim\limits_{z\rightarrow +\infty}g(z)=-\infty$.

Below, we will prove that there exist $d_1,\ldots,d_{s}$ such that $0\in (a, M)$,  $g'(0)=0$ and $g''(0)>0$.
For any  $i\in S_1$, let
$$d_i = \frac{|\gamma_i|}{X}, \;\;\; \text{where}\; X \; \text{is a real number.}$$
For any $i\in S_4\setminus \{\ell\}$, where $\ell$ can be any index chosen from $\argmin\limits_{k\in S_4}\{|\alpha_k|\}$, let
$$d_i = \frac{|\gamma_i|}{\epsilon}, \;\;\; \text{where}\; \epsilon \; \text{is a real number.}$$
Finally, for the index $\ell$ we choose in the previous step, let
$$d_{\ell}=\frac{|\gamma_{\ell}|}{Y},\;\;\; \text{where}\; Y \; \text{is a real number.}$$
Note here, for any $(\epsilon, X, Y)\in {\mathbb R}_{>0}^3$,
we have $0\in (a, M)=(\max\limits_{k\in S_1\cup S_4}\{-\frac{d_{k}}{\gamma_k}\}, +\infty)$.
Note also, $$\begin{aligned}
  &g'(0)=\sum\limits_{k\in S_1\cup S_4}\frac{\alpha_k\gamma_k}{\gamma_kz+d_k}|_{z=0}=\sum\limits_{k\in S_1}|\alpha_k|X-|\alpha_{\ell}|Y-\sum\limits_{k\in S_4\setminus\{l\}}|\alpha_k|\epsilon,\\
  &g''(0)=\sum\limits_{k\in S_1\cup S_4}\frac{\alpha_k\gamma^2_k}{(\gamma_kz+d_k)^2}|_{z=0}=-\sum\limits_{k\in S_1}|\alpha_k|X^2+|\alpha_{\ell}|Y^2+\sum\limits_{k\in S_4\setminus\{l\}}|\alpha_k|\epsilon^2.
\end{aligned}$$
If $\sum\limits_{k\in S_1}|\alpha_{k}|>\min\limits_{k\in S_4}\{|\alpha_k|\}$, i.e., $\sum\limits_{k\in S_1}|\alpha_{k}|>|\alpha_{\ell}|$, then for $\epsilon=0$,
for any positive number $X$, and for $Y=X\sum\limits_{k\in S_1}|\alpha_{k}|/|\alpha_{\ell}|$,
we have $g'(0)=0$ and $g''(0)>0$. So, when $\epsilon$ is a sufficiently small positive number,
there exists $(X, Y)\in {\mathbb R}_{>0}^2$ such that  $g'(0)=0$ and $g''(0)>0$.
So, we know by Lemma \ref{lm:3roots}  that we can choose  $d_1,\dots,d_{s}$
   such that for some real number $K$, $g(z)=K$ has at least 3 different solutions in the interval $I=(a, M)$.
\end{proof}

\begin{lemma}\label{lm:case1}
Suppose $cap_{pos}(G)<+\infty$.
  If there are exactly two of the four sets $S_1,S_2,S_3,S_4$ are non-empty, then the network $G$ admits at least 3 different positive steady states if and only if

  (1) both the sequence $\{\alpha_k\}_{k=1}^{i_4-1}$ and $\{\alpha_k\gamma_k\}_{k=1}^{i_4-1}$ change signs,

   (2) $\sum\limits_{i\in S_k}|\alpha_{i}|>\min\limits_{i\in S_{\ell}}\{|\alpha_i|\}$ and $\sum\limits_{i\in S_{\ell}}|\alpha_{i}|>\min\limits_{i\in S_k}\{|\alpha_i|\}$, where $S_k$ and $S_{\ell}$ are the two non-empty sets.
\end{lemma}

\begin{proof}

``$\Rightarrow$:" The statement (1) directly follows from Lemma \ref{lm:sign}.
If  there are exactly two of the four sets $S_1,S_2,S_3,S_4$ are non-empty, by the conclusion (1),
there are only two possibilities: $S_1$ and $S_4$ are non-empty,  or $S_2$ and $S_3$ are non-empty.  Without loss of generality, we assume $S_2=S_3=\emptyset$, $S_1\neq\emptyset$  and $S_4\neq\emptyset$.
 Thus, the statement (2) follows from Lemma \ref{lm:necessary}.

``$\Leftarrow$:"
Again, by the statement (1), we can assume that $S_1$ and $S_4$ are non-empty.
By Lemma \ref{lm:14}, we have $\sum\limits_{k\in S_1}|\alpha_k|\neq \sum\limits_{k\in S_4}|\alpha_k|$.
If $$\sum\limits_{k\in S_1}|\alpha_k|<\sum\limits_{k\in S_4}|\alpha_k|,\; \text{or}\; \sum\limits_{k\in S_1}|\alpha_k|>\sum\limits_{k\in S_4}|\alpha_k|,$$ then the conclusion follows from Lemma \ref{lm:sufficient}.
\end{proof}

\begin{lemma}\label{lm:g''2}
Suppose that  $S_3=\emptyset$, $S_1\neq\emptyset,S_2\neq\emptyset$, and $S_4\neq\emptyset$.
 Let $g(z)$ be the function defined in \eqref{eq:g}, and let $I$ be the interval defined in \eqref{eq:interval}.
 If there exists $z^*\in I$ such that $g'(z^*)=0$, then for $z=z^*$, we have the following inequality:
 \begin{align}\label{eq:g''2}
   g''(z)>\left(\sum\limits_{k\in S_1}|\alpha_k|b_k\right)^2\left(\frac{1}{\sum\limits_{k\in S_4}|\alpha_k|}-\frac{1}{\min\limits_{k\in S_1}\{|\alpha_k|\}}\right).
 \end{align}
\end{lemma}
\begin{proof}
 Since $ g'(z)=\sum\limits_{k\in S_1}|\alpha_k|b_k+\sum\limits_{k\in S_2}|\alpha_k|b_k-\sum\limits_{k\in S_4}|\alpha_k|b_k$, for $z=z^*$, $g(z^*)=0$ implies that
   \begin{align}\label{eq:g2I4}
 \sum\limits_{k\in S_4}|\alpha_k|b_k(z^*)= \sum\limits_{k\in S_1}|\alpha_k|b_k(z^*)+\sum\limits_{k\in S_2}|\alpha_k|b_k(z^*)>\sum\limits_{k\in S_1}|\alpha_k|b_k(z^*).
\end{align}
If $S_3=\emptyset$, then by \eqref{eq:cauchy}, \eqref{eq:cauchynext} and \eqref{eq:g2I4}, for $z=z^*$, we have
$$\begin{aligned}
     g''(z)&=-\sum\limits_{k\in S_1}|\alpha_k|b_k^2+\sum\limits_{k\in S_2}|\alpha_k|b_k^2+\sum\limits_{k\in S_4}|\alpha_k|b_k^2\\
     &>-\sum\limits_{k\in S_1}|\alpha_k|b_k^2+\sum\limits_{k\in S_4}|\alpha_k|b_k^2 \\
     &\geq -\frac{\left(\sum\limits_{k\in S_1}|\alpha_k|b_k\right)^2}{\min\limits_{k\in S_1}\{|\alpha_k|\}}
     +\frac{\left(\sum\limits_{k\in S_4}|\alpha_k|b_k\right)^2}{\sum\limits_{k\in S_4}|\alpha_k|} \\
     &> -\frac{\left(\sum\limits_{k\in S_1}|\alpha_k|b_k\right)^2}{\min\limits_{k\in S_1}\{|\alpha_k|\}}
     +\frac{\left(\sum\limits_{k\in S_1}|\alpha_k|b_k\right)^2}{\sum\limits_{k\in S_4}|\alpha_k|}.
   \end{aligned}$$
   So, we have the equality \eqref{eq:g''2}.
\end{proof}

\begin{lemma}\label{lm:necessary2}
  Suppose that $S_3=\emptyset$, $S_1\neq\emptyset,S_2\neq\emptyset$, and $S_4\neq\emptyset$. If the network $G$ admits at least 3 different positive steady states, then $\sum\limits_{i\in S_4}|\alpha_{i}|>\min\limits_{i\in S_1}\{|\alpha_i|\}$.
\end{lemma}
\begin{proof}
Let $g(z)$ be the function defined in \eqref{eq:g}. Let $I$ be the interval defined in \eqref{eq:interval}.
  Similar to the proof of Lemma \ref{lm:necessary}, if $G$ admits at least 3 different positive steady states, then there exist $d_1, \ldots, d_s$ such that for two different points $z_1, z_2\in I$, we have  $g'(z_1)=g'(z_2)=0$ and $g''(z_1)g''(z_2)\leq 0$. Without loss of generality, we assume that
  $g''(z_1)\leq 0$.
  By Lemma \ref{lm:g''2}, we have
\begin{align}
  \left(\sum\limits_{k\in S_1}|\alpha_k|b_k(z_1)\right)^2\left(\frac{1}{\sum\limits_{k\in S_4}|\alpha_k|}-\frac{1}{\min\limits_{k\in S_1}\{|\alpha_k|\}}\right) <g''(z_1)\leq 0.
 \end{align}
Therefore, we have $\min\limits_{i\in S_1}\{|\alpha_i|\}<\sum\limits_{i\in S_4}|\alpha_{i}|$.
\end{proof}

\begin{lemma}\label{lm:sufficient2}
  Suppose that $S_3=\emptyset$, $S_1\neq\emptyset,S_2\neq\emptyset$, and $S_4\neq\emptyset$. If $\sum\limits_{i\in S_4}|\alpha_{i}|>\min\limits_{i\in S_1}\{|\alpha_i|\}$, then the network $G$ admits at least 3 different positive steady states.
\end{lemma}

\begin{proof}
Similar to the proof of Lemma \ref{lm:necessary}, we only need to prove that we can choose
 $d_1, \ldots, d_s$ such that for some $K\in {\mathbb R}$,
 $g(z)=K$ has at least three different solutions in $I$.
 If  $S_3=\emptyset$, $S_1\neq\emptyset$,   $S_2\neq\emptyset$, and $S_4\neq\emptyset$, then by \eqref{eq:s} and \eqref{eq:am}, we have $a=\max\limits_{k\in S_1\cup S_4}\{-\frac{d_{k}}{\gamma_k}\}$ and $M=\min\limits_{k\in S_2}\{-\frac{d_{k}}{\gamma_k}\}$,
 and by \eqref{eq:g}, we have
\begin{align*}
  g(z)&=\sum\limits_{k\in S_1}|\alpha_k|\ln(\gamma_k z+d_k)-\sum\limits_{k\in S_2}|\alpha_k|\ln(\gamma_k z+d_k)-\sum\limits_{k\in S_4}|\alpha_k|\ln(\gamma_k z+d_k)
  \end{align*}
So,  we have $\lim\limits_{z\rightarrow M}g(z)=+\infty$.

Below, we will prove that there exist $d_1,\ldots,d_{s}$ such that $0\in (a, M)$,  $g'(0)=0$ and $g''(0)<0$.
For any $i\in (S_1\cup S_2)\setminus \{\ell\}$, where $\ell$ can be any index chosen from $\argmin\limits_{k\in S_1}\{|\alpha_k|\}$, let
$$d_i = \frac{|\gamma_i|}{\epsilon}, \;\;\; \text{where}\; \epsilon \; \text{is a real number.}$$
For the index $\ell$ we choose in the previous step, let
$$d_{\ell}=\frac{|\gamma_{\ell}|}{X},\;\;\; \text{where}\; X\; \text{is a real number.}$$
For any  $i\in S_4$, let
$$d_i = \frac{|\gamma_i|}{Y}, \;\;\; \text{where}\; Y \; \text{is a real number.}$$

Note here, for any $(\epsilon, X, Y)\in {\mathbb R}_{>0}^3$,
we have $$0\in (a, M)=(\max\limits_{k\in S_1\cup S_4}\{-\frac{d_{k}}{\gamma_k}\}, \min\limits_{k\in S_2}\{-\frac{d_{k}}{\gamma_k}\}).$$
Note also,
{\footnotesize
$$\begin{aligned}
  &g'(0)=\sum\limits_{k\in S_1\cup S_2\cup S_4}\frac{\alpha_k\gamma_k}{\gamma_kz+d_k}|_{z=0}=|\alpha_{\ell}|X-\sum\limits_{k\in S_4}|\alpha_k|Y+\sum\limits_{k\in S_1\cup S_2\setminus\{l\}}|\alpha_k|\epsilon,\\
  &g''(0)=\sum\limits_{k\in S_1\cup S_2\cup S_4}\frac{\alpha_k\gamma^2_k}{(\gamma_kz+d_k)^2}|_{z=0}=-|\alpha_{\ell}|X^2+\sum\limits_{k\in S_4}|\alpha_k|Y^2-
  (\sum\limits_{k\in S_1\setminus\{\ell\}}|\alpha_k|-\sum\limits_{k\in S_2}|\alpha_k|)\epsilon^2.
     \end{aligned}$$
     }
If $\sum\limits_{k\in S_4}|\alpha_{k}|>\min\limits_{k\in S_1}\{|\alpha_k|\}$, i.e., $\sum\limits_{k\in S_4}|\alpha_{k}|>|\alpha_{\ell}|$, then for $\epsilon=0$,
for any positive number $X$, and for $Y=X|\alpha_{\ell}|/\sum\limits_{k\in S_4}|\alpha_{k}|$,
we have $g'(0)=0$ and $g''(0)<0$. So, when $\epsilon$ is a sufficiently small positive number,
there exists $(X, Y)\in {\mathbb R}_{>0}^2$ such that  $g'(0)=0$ and $g''(0)<0$.
So, we know by Lemma \ref{lm:3roots}  that we can choose  $d_1,\dots,d_{s}$
   such that for some real number $K$, $g(x)=K$ has at least 3 different solutions in the interval $I=(a, M)$.
\end{proof}

\begin{lemma}\label{lm:case2}
  Suppose exactly three of the four sets $S_1, S_2, S_3, S_4$ are non-empty, say $S_j$, $S_k$, and $S_{\ell}$. Then, the network $G$ admits at least 3 different positive steady states if and only if $\sum\limits_{i\in S_{\ell}}|\alpha_{i}|>\min\limits_{i\in S_k}\{|\alpha_i|\}$, where $k$ and $\ell$ are assumed to be the two indices such that $sign_{\alpha}(k)=-sign_{\alpha}({\ell})=-sign_{\alpha}(j)$, and $sign_{\alpha\gamma}(\ell)=-sign_{\alpha\gamma}(k)=-sign_{\alpha\gamma}(j)$.
\end{lemma}
\begin{proof}
 Without loss of generality, let $S_3=\emptyset$, $S_1\neq\emptyset,S_2\neq\emptyset, S_4\neq\emptyset$. From $sign(\alpha_k)=-sign(\alpha_{\ell})=-sign(\alpha_j)$, and $sign(\alpha_{\ell}\gamma_{\ell})=-sign(\alpha_k\gamma_k)=-sign(\alpha_j\gamma_j)$, we get $k=1$ and $\ell=4$.
So the conclusion follows from Lemma \ref{lm:necessary2} and Lemma \ref{lm:sufficient2}.
\end{proof}

\begin{lemma}\label{lm:case31}
  If all the four sets $S_1, S_2, S_3, S_4$ are singleton, i.e., $S_k=\{k\}$ for $k\in \{1, 2, 3, 4\}$, and if
  $|\alpha_1|=|\alpha_4|$ and $|\alpha_2|=|\alpha_3|$,  then the network $G$ admits at most $2$ different positive steady states.
\end{lemma}
\begin{proof}
Let $g(z)$ be the function defined in \eqref{eq:g}. Let $I$ be the interval defined in \eqref{eq:interval}.
If $S_k=\{k\}$ for $k\in \{1, 2, 3, 4\}$, and if
  $|\alpha_1|=|\alpha_4|$ and $|\alpha_2|=|\alpha_3|$, then we have
 \begin{align}\label{eq:case31}
 g'(z)=|\alpha_1|b_1 + |\alpha_2|b_2 - |\alpha_2|b_3 - |\alpha_1|b_4.
 \end{align}
 Here, recall \eqref{eq:ab} that $b_i := |\frac{\gamma_i}{\gamma_iz+d_i}|$.

 {\bf Case 1.}
 For any $(d_2, d_3)\in {\mathbb R}^2$ such that $\frac{d_2}{\gamma_2}=\frac{d_3}{\gamma_3}$, we have
  \begin{align*}
 g'(z)=|\alpha_1|b_1  - |\alpha_1|b_4.
 \end{align*}
 Clearly, $g'(z)=0$ has at most one solution in $I$, and hence, for any $K\in {\mathbb R}$,
 $g(z)=K$ has at most two solutions in $I$. By Lemma \ref{lm:g}, the $cap_{pos}(G)\leq 2$.

  {\bf Case 2.}
 For any $(d_2, d_3)\in {\mathbb R}^2$ such that $\frac{d_2}{\gamma_2}\neq \frac{d_3}{\gamma_3}$,
 without loss of generality, we assume $\frac{d_2}{\gamma_2}<\frac{d_3}{\gamma_3}$. Then, for any
 $z\in I$, we have $b_2<b_3$.
 Note that
  \begin{align}
 g''(z)&=-|\alpha_1|b^2_1 +|\alpha_2|b^2_2 - |\alpha_2|b^2_3   + |\alpha_1|b^2_4 \notag\\
       &=-|\alpha_1|(b_1+b_4)(b_1-b_4) +  |\alpha_2|(b_2 +b_3)(b_2 -b_3). \label{eq:case311}
 \end{align}
 Suppose $z^*\in I$ and  $g'(z*)=0$. By \eqref{eq:case31}, we have
  \begin{align}\label{eq:case312}
       -|\alpha_1|(b_1(z^*)-b_4(z^*)) =  |\alpha_2|(b_2(z^*) -b_3(z^*)).
 \end{align}
 Substitute \eqref{eq:case312} into \eqref{eq:case311}, and we have
 \begin{align*}
 g''(z^*)= |\alpha_2|(b_1(z^*)+b_2(z^*) +b_3(z^*)+b_4(z^*))(b_2(z^*) -b_3(z^*))<0.
 \end{align*}
 That means $ g''(z)$ does not change sign on the set $\{z^*\in I|g'(z^*)=0\}$.
 So, $g'(z)=0$ has at most one solution in $I$, and hence, for any $K\in {\mathbb R}$,
 $g(z)=K$ has at most two solutions in $I$. By Lemma \ref{lm:g}, the $cap_{pos}(G)\leq 2$.
\end{proof}

\begin{lemma}\label{lm:case32}
  If all the four sets $S_1, S_2, S_3, S_4$ are non-empty, and if $\min\limits_{i\in S_1}\{|\alpha_i|\}<\sum\limits_{i\in S_4}|\alpha_{i}|$, then the network $G$ admits at least 3 different positive steady states.
\end{lemma}
\begin{proof}
 The proof is similar to the proof of  Lemma \ref{lm:necessary2}.
 If  $S_k\neq \emptyset$ for every $k\in \{1, 2, 3, 4\}$, then by \eqref{eq:s} and \eqref{eq:am}, we have $a=\max\limits_{k\in S_1\cup S_4}\{-\frac{d_{k}}{\gamma_k}\}$ and $M=\min\limits_{k\in S_2\cup S_3}\{-\frac{d_{k}}{\gamma_k}\}$,
 and by \eqref{eq:g}, we have
 {\scriptsize
\begin{align}\label{eq:gcase32}
  g(z)&=\sum\limits_{k\in S_1}|\alpha_k|\ln(\gamma_k z+d_k)-\sum\limits_{k\in S_2}|\alpha_k|\ln(\gamma_k z+d_k)+\sum\limits_{k\in S_3}|\alpha_k|\ln(\gamma_k z+d_k)-\sum\limits_{k\in S_4}|\alpha_k|\ln(\gamma_k z+d_k)
  \end{align}
  }

Below, we will prove that there exist $d_1,\ldots,d_{s}$ such that $\lim\limits_{z\rightarrow M}g(z)=+\infty$,  $0\in (a, M)$,  $g'(0)=0$ and $g''(0)<0$.
For any $i\in (S_1\cup S_2)\setminus \{\ell\}$, where $\ell$ can be any index chosen from $\argmin\limits_{k\in S_1}\{|\alpha_k|\}$, let
$$d_i = \frac{|\gamma_i|}{\epsilon_1}, \;\;\; \text{where}\; \epsilon_1 \; \text{is a real number.}$$
For the index $\ell$ we choose in the previous step, let
$$d_{\ell}=\frac{|\gamma_{\ell}|}{X},\;\;\; \text{where}\; X\; \text{is a real number.}$$
For any  $i\in S_3$, let
$$d_i = \frac{|\gamma_i|}{\epsilon_2}, \;\;\; \text{where}\; \epsilon_2 \; \text{is a real number}.$$
For any  $i\in S_4$, let
$$d_i = \frac{|\gamma_i|}{Y}, \;\;\; \text{where}\; Y \; \text{is a real number.}$$

Note here, for any $(\epsilon_1, \epsilon_2, X, Y)\in {\mathbb R}_{>0}^4$ such that $\epsilon_1>\epsilon_2$,
we have $$0\in (a, M)=(\max\limits_{k\in S_1\cup S_4}\{-\frac{d_{k}}{\gamma_k}\}, \min\limits_{k\in S_2}\{-\frac{d_{k}}{\gamma_k}\}).$$
So, by \eqref{eq:gcase32}, we have $\lim\limits_{z\rightarrow M}g(z)=+\infty$.
Note also, $$\begin{aligned}
  &g'(0)=|\alpha_{\ell}|X-
     \sum\limits_{k\in S_4}|\alpha_k|Y+(\sum\limits_{k\in S_1\setminus\{\ell\}}|\alpha_k|+\sum\limits_{k\in S_2}|\alpha_k|)\epsilon_1-\sum\limits_{k\in S_3}|\alpha_k|\epsilon_2,\\
  &g''(0)=-|\alpha_{\ell}|X^2+
    \sum\limits_{k\in S_4}|\alpha_k|Y^2-(\sum\limits_{k\in S_1\setminus\{\ell\}}|\alpha_k|-
     \sum\limits_{k\in S_2}|\alpha_k|)\epsilon_1^2-\sum\limits_{k\in S_3}|\alpha_k|\epsilon_2^2.
     \end{aligned}$$
If $\sum\limits_{k\in S_4}|\alpha_{k}|>\min\limits_{k\in S_1}\{|\alpha_k|\}$, i.e., $\sum\limits_{k\in S_4}|\alpha_{k}|>|\alpha_{\ell}|$, then for $\epsilon_1=\epsilon_2=0$,
for any positive number $X$, and for $Y=X|\alpha_{\ell}|/\sum\limits_{k\in S_4}|\alpha_{k}|$,
we have $g'(0)=0$ and $g''(0)<0$. So, when $\epsilon_1$ and $\epsilon_2$ are sufficiently small such that $\epsilon_1>\epsilon_2$,
there exists $(X, Y)\in {\mathbb R}_{>0}^2$ such that  $g'(0)=0$ and $g''(0)<0$.
So, we know by Lemma \ref{lm:3roots}  that we can choose  $d_1,\dots,d_{s}$
   such that for some real number $K$, $g(z)=K$ has at least 3 different solutions in the interval $I=(a, M)$.
   By Lemma \ref{lm:g}, we have $cap_{pos}(G)\geq 3$.
\end{proof}

\begin{lemma}\label{lm:case3}
  If all the four sets $S_1, S_2, S_3, S_4$ are non-empty, then the network $G$ admits at least 3 different positive steady states if and only if at least one of the four conditions \eqref{eq:cond1}--\eqref{eq:cond2} holds.
\end{lemma}

\begin{proof}
``$\Rightarrow$:"
We assume that none of the four conditions in \eqref{eq:cond1}--\eqref{eq:cond2} holds, i.e., we have the following four conditions
 \begin{align}
\min\limits_{i\in S_1}\{|\alpha_i|\}\geq \sum\limits_{i\in S_4}|\alpha_{i}|,\;\;
 \min\limits_{i\in S_4}\{|\alpha_i|\}\geq\sum\limits_{i\in S_1}|\alpha_{i}|, \label{eq:con1}\\
 \min\limits_{i\in S_2}\{|\alpha_i|\}\geq \sum\limits_{i\in S_3}|\alpha_{i}|,\;\;
 \min\limits_{i\in S_3}\{|\alpha_i|\}\geq \sum\limits_{i\in S_2}|\alpha_{i}|\label{eq:con4}
 \end{align}
 hold simultaneously.
 By \eqref{eq:con1}--\eqref{eq:con4},
 we have
 $$\min\limits_{i\in S_1}\{|\alpha_i|\}\geq\sum\limits_{i\in S_4}|\alpha_i|\geq \min\limits_{i\in S_4}\{|\alpha_i|\}\geq \sum\limits_{i\in S_1}|\alpha_i|\geq \min\limits_{i\in S_1}\{|\alpha_i|\}, \;\text{and}$$
 $$\min\limits_{i\in S_2}\{|\alpha_i|\}\geq\sum\limits_{i\in S_3}|\alpha_i|\geq \min\limits_{i\in S_3}\{|\alpha_i|\}\geq \sum\limits_{i\in S_2}|\alpha_i|\geq \min\limits_{i\in S_2}\{|\alpha_i|\}. $$
 So, we conclude that
 all the four sets $S_1, S_2, S_3, S_4$ are singleton, and if we assume that $S_k=\{k\}$ for every $k\in \{1, 2, 3, 4\}$, then
 we have $|\alpha_1|=|\alpha_4|$ and $|\alpha_2|=|\alpha_3|$. By Lemma \ref{lm:case31}, the network $G$ admits at most $2$ different positive steady states.

``$\Leftarrow$:"
 Suppose at least one of the four conditions in \eqref{eq:cond1}--\eqref{eq:cond2} holds. Without loss of generality, we assume that
 the first condition in \eqref{eq:cond1} holds, i.e.,  we have $\sum\limits_{i\in S_4}|\alpha_{i}|>\min\limits_{i\in S_1}\{|\alpha_i|\}$.
 By Lemma \ref{lm:case32},  the network $G$ admits at least 3 different positive steady states.
\end{proof}

{\bf Proof of Theorem \ref{thm:main}.}
\begin{proof}
The conclusion (a) follows from Lemma \ref{lm:sign} (b). The statements (b), (c), and (d) are proved in Lemma \ref{lm:case1}, Lemma \ref{lm:case2} and Lemma \ref{lm:case3}.
\end{proof}

\section{Discussion}\label{sec:dis}

For the networks with one-dimensional stoichiometric subspaces that admit finitely many positive steady states, we remark again that admitting at least three positive steady states is a necessary condition for multistability. So, characterizing networks admitting more positive steady states is as important as characterizing multistationarity. However, by comparing Theorem \ref{thm:main} and Lemma \ref{lm:nondegmss}, one can see that the conditions for admitting three (nondegenerate) positive steady states are more complicated than those for admitting two (nondegenerate) positive steady states (note here, by Theorem \ref{thm:nc}, we see that Theorem \ref{thm:main} also characterize nondegenerate multistationarity). So, even for bi-reaction networks, directly characterizing multistability or characterizing at least four positive steady states should be a challenging task.
Also, from Theorem \ref{thm:Q6.1} and Theorem \ref{thm:ad3}, we want to ask if it is true that
$cap_{pos}(G)\geq N$ implies $Ad(G)\geq N$.


\end{document}